\documentclass{amsart}
\usepackage[utf8]{inputenc}
\usepackage{amsfonts}
\usepackage{hyperref}
\hypersetup{
pdftitle={},
pdfauthor={},
}
\usepackage{breakurl}

\usepackage{amsmath}
\usepackage{xcolor}
\usepackage{amsthm}
\usepackage{pdflscape}
\usepackage{pgfplots}
\usepackage{mathrsfs}
\usepackage{enumitem}
\usepackage{multicol}

\usepackage[capitalise]{cleveref}
\crefformat{equation}{(#2#1#3)}
\crefrangeformat{equation}{(#3#1#4--#5#2#6)}
\crefformat{enumi}{(#2#1#3)}
\crefrangeformat{enumi}{(#3#1#4--#5#2#6)}

\newtheorem{thm}{Theorem}[section]
\newtheorem*{thm*}{Theorem}
\newtheorem*{cla*}{Classification}
\newtheorem*{main*}{Main Result}
\newtheorem{cor}[thm]{Corollary}
\newtheorem*{cor*}{Corollary}
\newtheorem{prop}[thm]{Proposition}
\newtheorem{lem}[thm]{Lemma}

\newtheorem{corollary}[thm]{Corollary}

\newtheorem{obs}[thm]{Observation}

\theoremstyle{definition}
\newtheorem{defn}[thm]{Definition}

\newtheorem{ex}[thm]{Example}

\newtheorem{notn}[thm]{Notation}

\newtheorem{remark}[thm]{Remark}
\newtheorem{appr}[thm]{Approach}

\theoremstyle{remark}
\newtheorem{rmk}[thm]{Remark}

\usepackage{amsmath}	
\usepackage{amssymb}		
\usepackage{amsthm}		
\usepackage{setspace}	
\usepackage{array} 
\usepackage{float}

\usepackage[mathscr]{euscript} 
 \let\mathscr\relax
\usepackage[scr]{rsfso}

\newcommand{\la}{\langle}
\newcommand{\ra}{\rangle}

\newcommand{\ZZ}{\mathbb{Z}}

\newcommand{\F}{\mathbb{F}}

\newcommand{\fX}{\mathfrak{X}}

\newcommand{\fY}{\Upsilon}

\newcommand{\cC}{\mathcal{C}}

\DeclareMathOperator{\coker}{coker}

\newcolumntype{P}[1]{>{\centering\arraybackslash}p{#1}}		            
\usepackage[top=1in,bottom=1in,left=1.25in,right=1in]{geometry}

\newtheorem*{lemma*}{lemma}

\usepackage{graphicx}

\usepackage[all]{xy}
\newcommand{\im}{\hspace{1mm}\text{im}\hspace{1mm}}

\newcommand{\onto}{\twoheadrightarrow}

\DeclareMathOperator{\lcm}{lcm}
\usepackage{tikz-cd}
\usetikzlibrary{cd}
\usepackage{adjustbox}
\usepackage{caption}  
\usepackage{subcaption}  
\usepackage{breakurl}

\usepackage[maxbibnames=9,style=alphabetic]{biblatex}
\newcommand{\Cone}[1]{\operatorname{Cone}\left({#1}\right)}

\newcommand{\FF}{\mathbb{F}}

\newcommand{\GG}{\mathbb{G}}

\newcommand{\KK}{\mathbb{K}}

\DeclareMathOperator{\Img}{Im}

\DeclareMathOperator{\Span}{span}
\DeclareMathOperator{\reg}{reg}
\DeclareMathOperator{\pdim}{pd}

\usepackage{amsaddr} 

\usepackage{orcidlink} 

\title[DG Algebra Structures on Mapping Cones with an Application to Edge Ideals]{DG Algebra Structures on Mapping Cones with an Application to Edge Ideals}

\author{Hugh Geller \orcidlink{0000-0002-4012-6404}}
\address{Center for Naval Analyses, Arlington, Virginia 22201 U.S.A.}
\email{\href{mailto:geller.hugh@gmail.com}{geller.hugh@gmail.com}}

\author{Desiree Martin \orcidlink{0000-0002-9297-582X}}
\address{Mathematics Department, Syracuse University, Syracuse, New York 13244 U.S.A.}
\email{\href{mailto:dmarti02@syr.edu}{dmarti02@syr.edu}}

\author{Matthew Mastroeni \orcidlink{0009-0003-6847-1079}}
\address{Department of Mathematics and Physics, SUNY Polytechnic Institute, Utica, New York 13502 U.S.A.}
\email{\href{mailto:mastromn@sunypoly.edu}{mastromn@sunypoly.edu}}

\author{Henry Potts-Rubin \orcidlink{0000-0002-9864-911X}}
\address{Department of Mathematics, Bowdoin College, Brunswick, Maine 04011 U.S.A.}
\email{\href{mailto:h.pottsrubin@bowdoin.edu}{h.pottsrubin@bowdoin.edu}}



\pgfplotsset{compat=1.18}
\begin{document}

\begin{abstract}
    We generalize a construction of Herzog and Takayama pertaining to differential graded algebra structures on mapping cones.  Via this generalization, we provide the minimal free resolution of the edge ideal of a ``new" graph built from suspension over a vertex cover of an ``old" graph, and we discuss when this process preserves differential graded algebra structure on the respective resolutions.  
\end{abstract}

\maketitle

\section{Introduction}

Imported to commutative algebra by Avramov, Buchsbaum, and Eisenbud, to name but a few, differential graded (dg) algebras first appeared in algebraic topology.  These days, the question of whether the minimal free resolution of a module admits a dg algebra structure is a central one in homological commutative algebra, not least because having a dg algebra on hand has proven to be extremely useful.  Among other applications, these structures play a central role in deformation theory (see, e.g., \cite{Manetti,Iacono}) and provide a framework for many senses of duality, for example, by contributing to a formula for the dualizing module of a Gorenstein ring \cite{DWYER}.    

It is challenging to determine whether or not a minimal free resolution is a dg algebra or module.  Some recent advances include those of \c{S}ega and Sireeshan, who explore these structures when an exact zero divisor is at play \cite{sega}; VandeBogert, who shows that a generalization of the Taylor complex admits the structure of dg algebra \cite{keller}; and three of the authors of this paper, who give a complete classification of trees and cycles with edge ideals minimally resolved by dg algebras \cite{GMP}.  We address edge ideals in the present paper, as well (see Sections \ref{suspensionsection} and \ref{iterativeconstructionsection}).  For example, we show that the minimal free resolution of the edge ideal of a threshold graph admits the structure of a dg algebra (Corollary \ref{threshold}).

The main result of this paper is the following construction, which generalizes a theorem of Herzog and Takayama \cite{mappingcones}.

\begin{main*}[Theorem \ref{thm: dgmodule M}]
    Let $A$ be a graded commutative dg algebra and $M$ be a dg $A$-module with $M_i=0$ for all $i<0$. Suppose there is a homomorphism $\beta: M \to A$ such that the induced homomorphism $\theta: M \otimes M \to M$ given by $\theta(m \otimes n) = m\beta(n) - \beta(m)n$ is nullhomotopic via an $A$-linear homotopy $\mu: M \otimes M \to M$ satisfying $\beta \mu =0$ and:
    \begin{equation}
        \mu(\mu(m \otimes n) \otimes u) = -(-1)^{|m|}\mu(m \otimes \mu(n \otimes u))
    \end{equation}
    for all $m,n,u \in M^{\natural}$. Then $C=\Cone{\beta}$ is a dg algebra with multiplication defined by 
    \[
    (a,m)(b,n) = (ab,mb+(-1)^{|a|}an + (-1)^{|m|}\mu(m \otimes n))
    \]
    for all $(a,m), (b,n)\in C^{\natural}_{\geq 1}$. Moreover, $C$ is graded commutative if $\mu$ satisfies 
    \begin{equation}
        \mu(m \otimes n)= -(-1)^{|m||n|}\mu(n \otimes m)
    \end{equation}
       for all $m,n, \in M^{\natural}$ and strictly commutative if $\mu$ additionally satisfies 
       \begin{equation}
           \mu(m \otimes m) = 0
       \end{equation}
       for all $m \in M^{\natural}$ with $|m|$ even.
\end{main*}


The outline of sections is as follows.  In Section \ref{defs}, we provide the necessary definitions around dg algebras and dg modules, so that in Section \ref{cones} we may extend a result of Herzog and Takayama in \cite{mappingcones} concerning such structure on mapping cones (see Theorems \ref{thm: dgmodule M} and \ref{thm-Betadef}).  We use Theorem \ref{thm: dgmodule M} in Section \ref{apps} to show that the minimal free resolutions of ideals built from complete intersections admit the structure of a strictly commutative dg algebra, and we then adapt Theorem \ref{thm-Betadef} to the setting of edge ideals and suspension (Theorem \ref{suspensionres}).  In particular, we provide an iterative process which constructs infinitely many graphs whose edge ideals are minimally resolved by a strictly commutative dg algebra.   

\section{DG Modules over DG Algebras}\label{defs}

In this section, we collect the terminology on dg algebras and dg modules that we will need throughout the paper. Our notation mostly follows that of Avramov \cite{AvramovIFR}. However, we take the point of view that differential graded objects are algebras and modules first and chain complexes second so that our notation $M^{\natural}$ differs slightly from the usage in \cite{AvramovIFR}. 

Let $R$ be a commutative ring. By a \textbf{graded} $R$-\textbf{algebra}, we mean a $\ZZ$-graded ring $A$ with $A_i=0$ for all $i < 0$ and $A_0=R$. If $M$ is a graded $A$-module, write $|m|=i$ for $m \in M_i$. Write $M^{\natural} = \cup_i M_i$ to denote the set of all homogeneous elements of $M$, and write $M_{\geq i}^{\natural} = \cup_{j \geq i}M_i$ to denote the set of homogeneous elements of degree at least $i$. We say that a graded $R$-algebra $A$ is \textbf{graded commutative} if $(ab)=(-1)^{|a||b|}ba$ for all $a,b \in A^{\natural}$ and \textbf{strictly commutative} if it is graded commutative and $a^2=0$ for all $a\in A^{\natural}$ with $|a|$ odd. We note that every graded commutative algebra is strictly commutative if $R$ does not have characteristic 2.

A \textbf{differential} on a graded $A$-module $M$ is a graded homomorphism $\partial: M \to M$ of degree $-1$ satisfying $\partial^2=0$. If $M$ has a differential, then we may view $M^{\natural}$ as a chain complex of $R$-modules:
\[
M^{\natural}: \hspace{5mm} \cdots\rightarrow M_{i+1} \xrightarrow{\partial} M_{i} \xrightarrow{\partial} M_{i-1} \rightarrow \cdots
\]
A \textbf{differential graded} (dg) $R$-\textbf{algebra} is a graded $R$-algebra $A$ together with a differential $\partial: A \to A$ satisfying the graded Leibniz rule
\[
\partial(ab)=\partial(a)b+(-1)^{|a|}a\partial(b)
\]
for all $a,b\in A^{\natural}$. A \textbf{homomorphism of dg algebras} is a graded homomorphism $\alpha: A \to B$ that is also a chain map. 

Give a dg algebra $A$, a \textbf{left dg $A$-module} is a left graded $A$-module $M$ together with a differential $\partial^M: M \to M$ satisfying 
\[
\partial^M(am) = \partial^A(a)m+(-1)^{|a|}a\partial^M(m)
\]
for all $a\in A^{\natural}$ and $m \in M^{\natural}$. We say that a graded $R$-linear map $\beta: M \to N$ of degree $s$ between left dg $A$-modules is $A$-\textbf{linear} if 
\[
\beta(am)=(-1)^{|a|s}a\beta(m)
\]
for all $m \in M^{\natural}$ and $a\in A^{\natural}$. If, in addition, we have $\beta\partial^M = (-1)^{s}\partial^M\beta$, then $\beta$ is called a \textbf{chain map}. For any left dg $A$-module $M$ and $s \in \ZZ$, the \textbf{shift} $M[s]$ is also a dg module with differential $\partial_{M[s]}=(-1)^{s}\partial^M$ and $A$-module structure given by $a\cdot_{[s]}m=(-1)^{|a|s}am$ for all $a \in A^{\natural}$ and $m \in M^{\natural}$. We note that an $A$-linear map (or chain map) $\beta: M \to N$ of degree $s$ is the same as an $A$-linear (or chain map) $\beta: M \to N[s]$ of degree 0. We will call an $A$-linear chain map of degree 0 a \textbf{homomorphism of dg modules}.

While the notion of a left dg module would be sufficient for our purposes, it will be convenient to also introduce right dg modules, as the sign convention for right modules is simpler. A \textbf{right dg $A$-module} is a right graded $A$-module $U$ together with a differential $\partial_U: U \to U$ satisfying
\[
\partial_U(ua) = \partial_U(u)a+(-1)^{|u|}u\partial^A(a)
\]
for all $a\in A^{\natural}$ and $u \in U^{\natural}$. When $A$ is graded commutative, then every left dg $A$-module $M$ also has the structure of a right dg $A$-module given by 
\[
ma = (-1)^{|a||m|}am
\]
for all $a \in A^{\natural}$ and $m \in M^{\natural}$, and conversely, every right module has a canonical left module structure. We say that a graded $R$-linear map $\beta: U \to V$ of degree $s$ between right dg $A$-modules is \textbf{$A$-linear} if 
\[
\beta(ua)=\beta(u)a
\]
for all $u \in U^{\natural}$ and $a \in A^{\natural}$. For any right dg $A$-module $U$ and $s\in \ZZ$, the \textbf{shift} $U[s]$ is also a dg module with differential $\partial_{U[s]} = (-1)^{s}\partial_U$ and $A$-module structure given by $u\cdot_{[s]}a = ua$ for all $u \in U^{\natural}$ and $a \in A^{\natural}$.

When $A$ is graded commutative and $M$ and $N$ are both dg $A$-modules, the \textbf{tensor product} of $M$ and $N$ is the dg $A$-module $M\otimes N$ obtained as a quotient of $M\otimes_R N$ by the subcomplex generated by the elements
\[
ma\otimes n - m \otimes an
\]
for all $a \in A^{\natural}$, $m\in M^{\natural}$ and $n \in N^{\natural}$. The $A$-module structure on $M\otimes N$ is given by 
\[
(m\otimes n)a = m\otimes na = (-1)^{|a||n|}m \otimes an = (-1)^{|a||n|}ma \otimes n.
\]
An $R$-bilinear map $\nu: M \times N \to U$ of degree $s$ between dg $A$-modules is $A$-\textbf{bilinear} if 
\[
\nu(m,na)=\nu(m,n)a = (-1)^{|a||n|}\nu(ma,n).\]
We note that the latter equality is equivalent to 
\[
\nu(am,n) = (-1)^{|a|s}a\nu(m,n).
\]

We end this section with two classic examples of strictly commutative dg algebras: the Taylor resolution and the Koszul complex.  

\begin{ex}[\cite{taylor, Gemeda}]\label{taylordef}
     Let $I$ be a monomial ideal of $Q=\Bbbk[x_1,\ldots,x_n]$ generated by the monomials $u_1, \ldots, u_t$, and let $<$ be a total order on these generators.  For $U \subseteq \{u_1,\ldots,u_t\}$, set $\lcm U=\lcm\{u_j \mid u_j \in U\}$.  The \textit{Taylor resolution} $\mathbb{T}$ of $Q/I$ is the $Q$-free resolution with
     \[
     \mathbb{T}_i=Q^{\binom{t}{i}},
     \]
     which has basis $g_U$, where $U$ is a subset of $\{u_1,\ldots,u_t\}$ of size $i$.  The differential of $\mathbb{T}$ is given by
     \[
     \partial(g_U)=\sum_{u \in U} (-1)^{\sigma(u,U)}\dfrac{\lcm_U}{\lcm_{U\backslash\{u\}}}g_{U \backslash\{u\}},
     \]
     where $\sigma(u,U)=|\{v \in U : v<u\}|$.  The Taylor resolution is often far-from-minimal, but it always admits the structure of strictly graded commutative dg algebra under the multiplication
     \[
g_V\cdot g_W=\begin{cases}
    (-1)^{\sigma(V,W)}\dfrac{\lcm_V\lcm_W}{\lcm_{V \cup W}}g_{V \cup W},& V \cap W=\emptyset,\\0,&V \cap W \neq \emptyset,
\end{cases}
\]
where $\sigma(V,W)=|\{(v,w)\in V \times W : v>w\}|$.  It is worth noting that any two total orders on the generators of $I$ induced isomorphic Taylor resolutions.      
\end{ex}

\begin{ex}
    The Koszul complex $\mathbb{K}$ of an ideal $I=(r_1,\ldots,r_n)$ of a commutative ring $R$ is the complex with free $R$-module $R^{{n \choose k}}$ in degree $k$, with basis elements $e_V$, where $V$ is a subset of $\{r_1,\ldots,r_n\}$ of size $k$.  The differential $\partial:R^{{n \choose k}} \to R^{{n \choose k-1}}$ is given by
    \[
    \partial(e_{\{r_{i_1},\ldots,r_{i_k}\}})=\sum_{j=1}^k(-1)^{j-1}r_{i_j}e_{\{r_{i_1},\ldots,r_{i_k}\}\smallsetminus\{r_{i_j}\}}.
    \]
   The Koszul complex is a strictly commutative dg algebra, and one can see this by viewing it as the exterior algebra $\bigwedge R^n$.  In particular, $e_V\cdot e_W=(-1)^{\sigma(V,W)}e_W \cdot e_V$.  Finally, we recall that if $r_1,\ldots,r_n$ is a regular sequence, then $\mathbb{K}$ minimally resolves $R/I$.  
\end{ex}

\section{DG Algebra Structures on Mapping Cones Via Idealization}\label{cones}

In \cite{mappingcones}, generalizing the notion on Nagata idealization  (or trivial extension) of a module over a commutative ring, Herzog and Takayama introduce the idealization  $A \ltimes M$ of a dg module $M$ over a graded commutative dg algebra $A$ admitting a homomorphism of complexes $\beta: M \to A$ satisfying the condition
\begin{equation}
\label{betamap}
    m \beta(n) = \beta(m)n
\end{equation}
 for all $m,n \in M^{\natural}$. Specifically, they show that $A \ltimes M$ is the mapping cone of $\beta$ together with a multiplication defined by 
\[
(a,m)(b,n) = (ab, mb + (-1)^{|a|}an)
\]
for all $(a,m)$, $(b,n) \in (A\ltimes M)^{\natural}$, and this generalizes an earlier construction of Avramov and Levin \cite{levin} in the case where $M$ is a dg ideal of $A$. The dg algebra structures of Koszul complexes, the Taylor resolution of a monomial ideal, and the minimal free resolution of Northcott ideals all arise via (possibly repeated) idealization \cite[Examples 3.5, 3.6, 3.7]{mappingcones}. In this section, we generalize Herzog and Takayama's construction by showing how to give a dg algebra structure on Cone($\beta$) when (\ref{betamap}) fails.

\begin{lem}
    Let $A$ be a graded commutative dg algebra and $M$ be a dg $A$-module with $M_i = 0$ for all $i < 0$ admitting a homomorphism $\beta:M\to A$. Then there is an induced homomorphism $\theta: M \otimes M \to M$ determined by 
    \[
    \theta(m \otimes n)= m \beta(n) - \beta(m)n
    \]
    for all $m,n \in M^{\natural}$.
\end{lem}

\begin{proof}
    It is easily seen that the assignment $(m,n) \mapsto  m \beta(n) - \beta(m)n $ is $A$-bilinear for all $m,n \in M^{\natural}$ so that there is a well-defined graded $A$-linear map $\theta: M \otimes M \to M$. To see that $\theta$ is a homomorphism of dg modules, we note that 
    \begin{align*}
        \theta \partial^{M\otimes M}(m \otimes n) &= \theta (\partial^M(m)\otimes n + (-1)^{|m|} m \otimes \partial^M(n))\\
        & = \partial^M(m)\beta(n)-\beta\partial^M(m)n+(-1)^{|m|}m\beta\partial^M(n) - (-1)^{|m|}\beta(m)\partial^M(n)\\
        &=\partial^M(m)\beta(n)+ (-1)^{|m|}m \partial^M \beta(n) - \partial^M \beta(m)n - (-1)^{|m|}\beta(m)\partial^M(n)\\
        &=\partial^M(m\beta(n))-\partial^M(\beta(m)n) = \partial^M\theta(m\otimes n)
    \end{align*}
    for all $m,n \in M^{\natural}$.
\end{proof}
Herzog and Takayama construct a dg algebra structure on the mapping cone of $\beta: M \to A$ under the assumption that the induced homomorphism $\theta: M \otimes M \to M$ is zero. We weaken this assumption to the case where $\theta$ is nullhomotopic.

\begin{thm}
\label{thm: dgmodule M}
    Let $A$ be a graded commutative dg algebra and $M$ be a dg $A$-module with $M_i=0$ for all $i<0$. Suppose there is a homomorphism $\beta: M \to A$ such that the induced homomorphism $\theta: M \otimes M \to M$ given by $\theta(m \otimes n) = m\beta(n) - \beta(m)n$ is nullhomotopic via an $A$-linear homotopy $\mu: M \otimes M \to M$ satisfying $\beta \mu =0$ and:
    \begin{equation}
    \label{twisting data}
        \mu(\mu(m \otimes n) \otimes u) = -(-1)^{|m|}\mu(m \otimes \mu(n \otimes u))
    \end{equation}
    for all $m,n,u \in M^{\natural}$. Then $C=\Cone{\beta}$ is a dg algebra with multiplication defined by 
    \[
    (a,m)(b,n) = (ab,mb+(-1)^{|a|}an + (-1)^{|m|}\mu(m \otimes n))
    \]
    for all $(a,m), (b,n)\in C^{\natural}_{\geq 1}$. Moreover, $C$ is graded commutative if $\mu$ satisfies 
    \begin{equation}
    \label{gradedcommutative condition}
        \mu(m \otimes n)= -(-1)^{|m||n|}\mu(n \otimes m)
    \end{equation}
       for all $m,n, \in M^{\natural}$ and strictly commutative if $\mu$ additionally satisfies 
       \begin{equation}
       \label{strictly commutative condition}
           \mu(m \otimes m) = 0
       \end{equation}
       for all $m \in M^{\natural}$ with $|m|$ even.
\end{thm}

\begin{rmk}
    The Nagata idealization of a module over a commutative ring is an example of a square-zero extension (see, e.g., \cite{extensions}).  Under the above multiplication, $C$ is a square-zero extension of $A$ by $M$ precisely when one of the following equivalent conditions holds: (i) $\mu$ is identically zero, (ii) $m\beta(n)=\beta(m)n$ for all $m,n, \in M^\natural$, (iii) $C$ is precisely the idealization $A \ltimes M$ in \cite{mappingcones}.  
\end{rmk}

We call the dg algebra defined by the preceding proposition the \textbf{twisted idealization} of $M$ by $\beta$ and $\mu$ and denote it by $A \ltimes_{\beta, \mu}M$. We also refer to the pair $\beta, \mu$ of a homomorphism of dg modules $\beta: M  \to A$ and homotopy $\mu: M \otimes M \to M$ satisfying condition (\ref{twisting data}) of the preceding theorem as \textbf{twisting data} for $M$. 

\begin{proof}[Proof of Theorem \ref{thm: dgmodule M}]
    Let $(a,m), (b,n), (c,u) \in C_{\geq 1}$. It is easily verified that the multiplication in the proposition (together with the natural $R$-module structures) satisfies the conditions for $C$ to be a graded $R$-algebra except possibly associativity. Before we begin, recall that $|a|=|m|-1$ and $|b|=|n|-1$ as we will use these regularly. To check that the multiplication of $C$ is associative, we compute:
    \begin{align*}
        [(a,m)(b,n)](c,u)&= (ab, mb+(-1)^{|a|}an +(-1)^{|m|}\mu(m \otimes n))(c,u)\\
        & = (abc, mbc+(-1)^{|a|}anc + (-1)^{|m|}\mu(m \otimes n)c +(-1)^{|a|+|b|}abu\\
        &\hspace{8mm}+(-1)^{|m|+|b|}\mu(mb\otimes u) + (-1)^{|m|+|b|+|a|}\mu(an\otimes u) + (-1)^{|b|}\mu(\mu(m\otimes n) \otimes u))\\
         & = (abc, mbc+(-1)^{|a|}anc + (-1)^{|m|}\mu(m \otimes n)c +(-1)^{|a|+|b|}abu\\
        &\hspace{8mm}+(-1)^{|m|+|b|}\mu(mb\otimes u) + (-1)^{|n|}\mu(an\otimes u) + (-1)^{|b|}\mu(\mu(m\otimes n) \otimes u)).\\
    \end{align*}
    On the other hand,
     \begin{align*}
        (a,m)[(b,n)(c,u)]&= (a,m)(bc, nc+(-1)^{|b|}bu +(-1)^{|n|}\mu(n \otimes u))\\
        & = (abc, mbc+(-1)^{|a|}anc + (-1)^{|a|+|b|}abu + (-1)^{|n|+|a|}a\mu(n\otimes u)\\
        &\hspace{8mm}+(-1)^{|m|}\mu(m \otimes nc) +(-1)^{|m|+|b|}\mu(m\otimes bu) + (-1)^{|m|+|n|}\mu(m \otimes\mu(n\otimes u)))\\
         & = (abc, mbc+(-1)^{|a|}anc + (-1)^{|a|+|b|}abu + (-1)^{|n|}\mu(an\otimes u)\\
        &\hspace{8mm}+(-1)^{|m|}\mu(m \otimes n)c +(-1)^{|m|+|b|}\mu(mb\otimes u) + (-1)^{|n|-1}\mu(\mu(m\otimes n) \otimes u))),\\
    \end{align*}
    showing that the multiplication is associative since $\mu$ is $A$-linear and satisfies condition (\ref{twisting data}).

    To see that $C$ is a dg algebra under this multiplication, we compute:
    \begin{align*}
        \partial^C((a,m)(b,n)) & = \left(\partial^A(a) + \beta(m), -\partial^M(m)(b,n) \right)\\
        & = (\partial^A(a)b  +\beta(m)b,\\
        &\hspace{15mm}- \partial^M(m)b  - (-1)^{|a|}\partial^A(a)n - (-1)^{|a|}\beta(m)n + (-1)^{|m|}\mu(\partial^M(m) \otimes n))
    \end{align*}
    and 
    \begin{align*}
        (-1)^{|a|}(a,m)\partial^C(b,n) &= (-1)^{|a|}(a,m){\partial^A(b)+\beta(n),-\partial^M(n)}\\
        & = ((-1)^{|a|}a\partial^A(b)+(-1)^{|a|}a\beta(n),\\
        &\hspace{15mm} - (-1)^{|m|}m\partial^A(b)-(-1)^{|m|}m\beta(n)-a\partial^M(n)
+\mu(m \otimes \partial^M(n)))    \end{align*}
so that
\begin{align*}
     \partial^C((a,m)(b,n)) &+ (-1)^{|a|}(a,m)\partial^C(b,n)  = (\partial^A(ab) + 
     \beta(mb)+(-1)^{|a|}\beta(an),\\
     &-\partial^M(mb)-(-1)^{|a|}\partial^M(an)+(-1)^{|m|}\beta(m)n-(-1)^{|m|}m \beta(n)\\
     &\hspace{25mm} +(-1)^{|m|}\mu(\partial^M(m)\otimes n) + \mu(m \otimes \partial^M(n)).
\end{align*}
        Noting that $\beta \mu =0$ and 
        \[
        (-1)^{|m|}m \beta(n) - (-1)^{|m|}\beta(m)n = (-1)^{|m|}\partial^M\mu(m \otimes n) + (-1)^{|m|}\mu(\partial^M(m)\otimes n) + \mu(m \otimes \partial^M(n))
        \]
        we see that 
        \begin{align*}
            \partial^C((a,m)(b,n)) + (-1)^{|a|}(a,m)\partial^C(b,n) = &(\partial^A(ab) + 
     \beta(mb)+(-1)^{|a|}\beta(an) + (-1)^{|m|}\beta\mu(m \otimes u),\\
     &-\partial^M(mb) -(-1)^{|a|}\partial^M(an)-(-1)^{|M|}\partial^M\mu(m \otimes n)\\
     & \hspace{45mm} = \partial^C((a,m)(b,n)).
        \end{align*}
        Hence, $C$ is a dg algebra.

        Finally, if $A$ is graded commutative and $\mu$ satisfies condition (\ref{gradedcommutative condition}), then
        \begin{align*}
            (-1)^{|a||b|}(b,n)(a,m) &= (-1)^{|a||b|}(ba, na +(-1)^{|b|}bm +(-1)^{n}\mu(n\otimes m )\\
            &= (ab, (-1)^{|a||n|+|a|}na + (-1)^{|b||m|bm}+(-1)^{|m||n|+|m|+1}\mu(n \otimes m))\\
            &=(ab, (-1)^{|a|}an+mb+ (-1)^{|m|}\mu(m \otimes n)) = (a,m)(b,m)
        \end{align*}
        which shows $C$ is graded commutative. If $\mu$ also satisfies (\ref{strictly commutative condition}) and $|(a,m)|=|a|$ is odd so that $|m|$ is even, then 
        \[
        (a,m)(a,m) (a^2, ma-am + \mu(m \otimes m)) = (a^2, ma-ma +\mu(m \otimes m))=(0,0)
        \]
        showing $C$ is strictly commutative.
\end{proof}

Our next theorem shows how to construct plenty of dg $A$-modules $M$ with a homomorphism $\beta: M \to A$ to which the preceding theorem can be applied. In general, we note that the induced homomorphism $\theta: M \otimes M \to M$ will be nonzero for these examples: see (\ref{mu definition base}).

\begin{thm}
\label{thm-Betadef}
    Let $\alpha: A \to B$ be a homomorphism of graded commutative dg $R$-algebras.
    \begin{enumerate}[label=(\alph*)]
        \item The complex $M= \Cone{\alpha_{\geq 1}}[1]$ is a dg $A$-module with multiplication defined by 
        \[
        am = \begin{cases}
            ((-1)^{|a|}\alpha(a)b,ac), & m = (b,c) \in M_{\geq 1}^{\natural}\\
            
            ((-1)^{|a|}\alpha(a)m,-a \partial^B(m)), & m \in M_0 = B_1, a \in A_{\geq 1}^{\natural}
        \end{cases}
        \]
        
        \item There is a homomorphism $\beta: M \to A$ given by
         \[
        \beta(m) = \begin{cases} -c, & m= (b,c) \in M_{\geq 1}^{\natural}\\
        \partial^B(m), & m \in M_0 = B_1 
        \end{cases}
        \]

        \item There is an $A$-linear nullhomotopy $\mu: M \otimes M \to M$ for the induced homomorphism $\theta: M \otimes M \to M$ defined by $\theta(m \otimes n) = m\beta(n)-\beta(m)n$, where 
        \[\mu(m \otimes n) = 
        \begin{cases}
            (-(-1)^{|b|}bd, 0), & m=(b,c), n=(d,e) \in M_{\geq 1}^{\natural}\\
            (md,0), & m \in M_0, n=(d,e) \in M_{\geq 1}^{\natural}\\
            (-(-1)^{|b|}bn,0), & n \in M_0, m=(b,c) \in M_{\geq 1}^{\natural}\\
            (mn,0), & m,n \in M_0 = B_1
        \end{cases}
        \]
        and $\mu$ satisfies $\beta\mu =0$ and conditions (\ref{twisting data}) and (\ref{gradedcommutative condition}). Moreover, if $B$ is strictly commutative, then $\mu$ also satisfies (\ref{strictly commutative condition}).

        
    \end{enumerate}
\end{thm}

\begin{proof}
    \begin{enumerate}[label=(\alph*)]
        \item It is easily checked that the operation defined in the theorem (together with the natural $R$-module structure on $M$) is $R$-bilinear on homogeneous elements of $A$ and $M$, and so, it extends to an $R$-linear map $A\otimes_R M \to M$ making $M$ into a left graded $A$-module. Indeed, let $a,f \in A_{\geq 1}^{\natural}$. If $m =(b,c) , n=(d,e) \in M_{\geq 1}^{\natural}$ with $|m|=|n|$, then:
    \end{enumerate}

    \begin{align*}
        (af)m & = ((-1)^{|a|+|f|}\alpha(af)b, afc)=((-1)^{|a|+|f|}\alpha(a)\alpha(f)b,afc)\\
        &=a((-1)^{|f|}\alpha(f)b,fc)=a(fm)
    \end{align*}
    \begin{align*}
        a(m+n) &= a(b+d,c+e) = ((-1)^{|a|}\alpha(a)(b+d),a(c+e))\\
        &= ((-1)^{|a|}\alpha(a)c,ac) + ((-1)^{|a|}\alpha(a)d,ae) = am+an
    \end{align*}
    and if $|a| =|f|$, then 
    \begin{align*}
        (a+f)m &= ((-1)^{|a|}\alpha(a+f)b,(a+f)c)\\
        &=((-1)^{|a|}\alpha(a)b,ac) +((-1)^{|f|}\alpha(f)b,fc) = am+fm.
    \end{align*}
    If $m,n \in M_0$, then:
\begin{align*}
    (af)m &= ((-1)^{|a|+|f|}\alpha(af)m, -af\partial^B(m)) = ((-1)^{|a|+|f|}\alpha(a)\alpha(f)m, -af\partial^B(m))\\
    & = a((-1)^{|f|}\alpha(f)m, - f\partial^B(m)) = a(fm)
\end{align*}
\begin{align*}
    a(m+n) & = ((-1)^{|a|}\alpha(a)(m+n),-a\partial^B(m+n))\\
    & = ((-1)^{|a|}\alpha(a)m,-a\partial^B(m) )+ ((-1)^{|a|}\alpha(a)n,-a\partial^B(n)) = am+an
\end{align*}
and if $|a|=|f|$, then 
\begin{align*}
    (a+f)m &= ((-1)^{|a|}\alpha(a+f)m, -(a+f)\partial^B(m))\\
    & = ((-1)^{|a|}\alpha(a)m,-a\partial^B(m)) + ((-1)^{|f|}\alpha(f)m,-f\partial^B(m)) = am+fm.
\end{align*}

It remains to show that this multiplication makes $M$ into a dg $A$-module. If $m=(b,c) \in M_{\geq 1}^{\natural}$, then:
\begin{align*}
    \partial^M(am) &= (-(-1)^{|a|}\partial^B(\alpha(a)b) - \alpha(ac), \partial^A(ac))\\
    &=(-(-1)^{|a|}\partial^B\alpha(a)b- \alpha(a)\partial^B(b)-\alpha(a)\alpha(c), \partial^A(a)c+(-1)^{|a|}a\partial^A(c))\\
    & = ((-1)^{|a|-1}a \partial^A(a)b-\alpha(a)\partial^B(b)-\alpha(a)\alpha(c), \partial^A(a)c+(-1)^{|a|}a\partial^A(c))\\
    & =((-1)^{|a|-1}\alpha\partial^A(a)b, \partial^A(a)c) + (-\alpha(a)\partial^B(b)-\alpha(a)\alpha(c), (-1)^{|a|}a\partial^A(c))\\
    & = \partial^A(a)m+(-1)^{|a|}a(-\partial^B(b)-\alpha(c),\partial^A(c))\\
    &= \partial^A(a)m + (-1)^{|a|}a\partial^M(m).
    \end{align*}
    Otherwise, if $m\in M_0$, then:
    \begin{align*}
        \partial^M(am) &= (-(-1)^{|a|}\partial^B(\alpha(a)m)+\alpha(a\partial^B(m)),-\partial^A(a\partial^B(m)))\\
        & = ((-1)^{|a|-1}\partial^B\alpha(a)m-\alpha(a)\partial^B(m)+\alpha(a)\partial^B(m),-\partial^A(a)\partial^B(m))\\
        & = ((-1)^{|a|-1}\alpha\partial^A(a)m,-\partial^A(a)\partial^B(m))\\
        & = \partial^A(a)m= \partial^A(a)m + (-1)^{|a|}a\partial^M(m)
    \end{align*}
    since $\partial^M(m) = 0$. Hence, $M$ is a dg $A$-module.

\item The map $\beta: M \to A$ defined in the proposition is easily checked to be $R$-linear. We check simultaneously that $\beta$ is a chain map and $A$-linear. If $m = (b,c) \in M_{\geq 1}^{\natural}$ and $a\in A_{\geq 1}^{\natural}$, then
\begin{align*}
    \beta\partial^M(m) = \beta(-\partial^B(b)-\alpha(c), \partial^A(c))= -\partial^A(c) = \partial^A\beta(m)
\end{align*}
\begin{align*}
    \beta(am) = \beta((-1)^{|a|}\alpha(a)b,ac) = -ac = a\beta(m). 
\end{align*}
On the other hand, if $m \in M_0$ and $a \in A_{\geq 1}^{\natural}$, then
\[
\beta\partial^M(m)=0=\partial^A\beta(m)
\]
since $\beta(m)\in R=A_0$, and 
\[
\beta(am)=\beta((-1)^{|a|}\alpha(a)m, -a\partial^B(m)) = a\partial^B(m)=a\beta(m).
\]
Hence, $\beta$ is a homomorphism of dg $A$-modules.

\item First, we not that:
\[
ma =(-1)^{|a||m|}am = \begin{cases}
    (b\alpha(a),ca), & m =(b,c) \in M_{\geq 1}^{\natural}\\
    (m\alpha(a),-a\partial^B(m)), & m\in M_0=B_1, a\in A_{\geq 1 }^{\natural}
\end{cases}
\]
so that 
\begin{equation}
\label{mu definition base}
m\beta(n)-\beta(m)n
 = \begin{cases}
     ((-1)^{|c|}\alpha(c)d-b\alpha(e),0), & m=(b,c), n=(d,e) \in M_{\geq 1}^{\natural}\\
     (-m\alpha(e)-\partial^B(m)d, 0), & m \in M_0=B_1, n=(d,e)\in M_{\geq 1}^{\natural}\\
     (b\partial^B(n)-(-1)^{|c|}\alpha(c)n,0), & n \in M_0=B_1, m=(b,c)\in M_{\geq 1}^{\natural}\\
     m\partial^B(n)-\partial^B(m)n, & m,n\in M_0=B_1.
 \end{cases}\end{equation}
We will check that the assignment $(m,n) \mapsto \mu(m\otimes n)$ is $A$-bilinear for all $m,n \in M^{\natural}$ so that there is a well-defined $A$-linear map $\mu: M \otimes M \to M$ of degree $1$. In face, we will simultaneously verify that: (i) $(m,n) \mapsto \mu(m\otimes n)$ is $A$-linear, (ii) $\mu$ satisfies (\ref{gradedcommutative condition}),  (iii) $\mu$ satisfies (\ref{twisting data}), and (iv)
 $\mu$ acts as a nullhomotopy for $\theta$ on simple tensors.

Before verifying these properties, we note that in the presence of (\ref{gradedcommutative condition}) it is enough to show that $(m,n) \mapsto \mu(m \otimes n)$ is $A$-linear as a function of $n$ to show that it is bilinear. Indeed, under these conditions if $a\in A_{\geq 1}^{\natural}$ and $m,n \in M^{\natural}$, then:
\begin{align*}
    \mu(am \otimes n) &= -(-1)^{|a||n|+|m||n|}\mu(n \otimes am)\\
    & = -(-1)^{|a||n|+|m||n|+|a||m|}\mu(n \otimes ma)\\
    & = (-1)^{|a||n|+|m||n|+|a||m|}\mu(n \otimes m)a\\
    & = (-1)^{|a||n|+|a||m|}\mu(m \otimes n)a 
    = (-1)^{|a|}a\mu(m \otimes n).
\end{align*}
We now turn our attention to verifying (i)-(iv) in several cases.

\noindent CASE 1: $m = (b,c), n=(d,e) \in M_{\geq 1}^{\natural}$

\begin{itemize}
    \item[(i)] \[ \mu(m \otimes na) = (-(-1)^{|b|}bd\alpha(a),0) = (-(-1)^{|b|}bd,0)a = \mu(m\otimes n)a\]

    \item[(ii)] \begin{align*}
        \mu(m\otimes n) = (-(-1)^{|b|}bd,0) &= (-(-1)^{|b|+|b||d|}bd,0)\\
        &= (-1)^{|b||d|+|b|+|d|}(-(-1)^{|d|}db,0)\\
        & = -(-1)^{|m||n|}(-(-1)^{|d|}db,0) = -(-1)^{|m||n|}\mu(m\otimes n)
    \end{align*}
    \item[(iii)] If $u = (f,g)\in M_{\geq 1}^{\natural}$, then
    \begin{align*}
-(-1)^{|m|}\mu(m\otimes \mu(n \otimes u))&=(-1)^{|b|} ((-1)^{|b|+|d|}bdf,0)\\
        &= ((-1)^{|d|}bdf,0) = \mu(\mu(m\otimes n)\otimes u)
    \end{align*}
    and otherwise if $u \in M_0$, then
    \begin{align*}
        -(-1)^{|m|}\mu(m\otimes \mu(n \otimes u))&=(-1)^{|b|} ((-1)^{|b|+|d|}bdu,0)\\
        &= ((-1)^{|d|}bdu,0) = \mu(\mu(m\otimes n)\otimes u).
   \end{align*}
        \item [(iv) ]
        \begin{align*}
            \partial^M\mu&(m\otimes n) + \mu\partial^{M\otimes M}(m \otimes n) \\
            &= \partial^M(-(-1)^{|b|}bd,0)+ \mu(\partial^M(m)\otimes n)+(-1)^{|m|}\mu(m\otimes \partial^M(n))\\
            & = ((-1)^{|b|}\partial^B(b)d+b\partial^B(d), 0 ) + (-(-1)^{
            |b|}\partial^B(b)d-(-1)^{|b|}\alpha(c)d,0) + (-b\partial^B(d)-b\alpha(e),0)\\
            & = ((-1)^{|c|}\alpha(c)d-b\alpha(e),0) = m \beta(n)-\beta(m)n.
        \end{align*}         
\end{itemize}

\noindent CASE 2: $m\in M_0$ and $n = (d,e)\in M_{\geq 1}^{\natural}$

\begin{itemize}
    \item [(i)] \[ \mu(m\otimes na)=(md\alpha(a),0) = (md,0)a = \mu(m\otimes n)a
    \]
    \item [(ii)] 
    \begin{align*}
        \mu(m \otimes n) = (md,0) &= ((-1)^{2|d|}md,0) = ((-1)^{|d|}dm,0)\\
        &= -(-(-1)^{|d|}dm,0) = -(-1)^{|m||n|}\mu(m \otimes n)
    \end{align*}

    \item [(iii)]If $u= (f,g)\in M_{\geq 1}^{\natural} $, then
    \begin{align*}
        -(-1)^{|m|}\mu(m \otimes \mu(n \otimes u) )& = -(-1)^{|m|}(-(-1)^{|d|}mdf,0)\\
        & = (-(-1)^{|d|+1}mdf,0) = \mu(\mu(m \otimes n)\otimes u)
    \end{align*}
    and otherwise if $u \in M_0$, then
    \begin{align*}
        -(-1)^{|m|}\mu(m \otimes \mu(n \otimes u) ) & = -(-1)^{|m|}(-(-1)^{|d|}mdu,0)\\
        &=(-(-1)^{|d|+1}mdu,0) = \mu(\mu(m\otimes n)\otimes u).
    \end{align*}
    \item[(iv)]
    \begin{align*}
        \partial^M\mu(m\otimes n) + \mu\partial^{M\otimes M}(m \otimes n) & = \partial^M(md,0)+\mu(\partial^M(m)\otimes n) + \mu(m \otimes \partial^M(n))\\
        &= (-\partial^B(m)d+m\partial^B(d),0)+(-m\partial^B(d)-m\alpha(e),0)\\
        &= (-\partial^B(m)d-m\alpha(e),0) = m\beta(n)-\beta(m)n.
    \end{align*}
\end{itemize}

\noindent CASE 3: $m = (b,c)\in M_{\geq 1}^{\natural}$and  $n\in M_0$ 

\begin{itemize}
    \item [(i)] \[ \mu(m \otimes na) = (-(-1)^{|b|}bn\alpha(a),0) = (-(-1)^{|b|}bn,0)a = \mu(m\otimes n)a\]

    \item[(ii)] 
   \[
        \mu(m\otimes n)= (-(-1)^{|b|}bn,0) = -(nb,0) = -(-1)^{|m||n|}\mu(n\otimes m)
    \]
    \item[(iii)] If $u= (f,g) \in M_{\geq 1}^{\natural}$, then
    \begin{align*}
        -(-1)^{|m|}\mu(m\otimes \mu(n \otimes u))
        &= (-1)^{|b|}((-1)^{|b|+1}bnf)\\
        &= ((-1)^{2|b|+1}bfn,0) = \mu(\mu(m\otimes n) \otimes u)
    \end{align*}
    and otherwise if $u\in M_0$, then
    \begin{align*}
         -(-1)^{|m|}\mu(m\otimes \mu(n \otimes u))
        &= (-1)^{|b|}((-1)^{|b|+1}bnu)\\
        &= ((-1)^{2|b|+1}bfu,0) = \mu(\mu(m\otimes n) \otimes u).
    \end{align*}
    \item[(iv)]
    \begin{align*}
        \partial^M\mu(m\otimes n)+\mu\partial^{M\otimes M}(m \otimes n) &= \partial^M(-(-1)^{|b|}bn,0)+ \mu(\partial^M(m)\otimes n)+\mu(m \otimes \partial^M(n))\\
        &=((-1)^{|b|}\partial^B(b)n-b\partial^B(n),0)+((-1)^{|b|-1}\partial^B(b)n+(-1)^{|b|-1}\alpha(c)n,0)\\
        &= (-b\partial^B(n)+(-1)^{|c|}\alpha(c)n,0)= m\beta(n)-\beta(m)n.
    \end{align*}
\end{itemize}

\noindent CASE 4: $m,n\in M_0$

    \begin{itemize}
        \item [(i)] \[\mu(m \otimes na) = (mn\alpha(a),0)=(mn,0)a = \mu(m\otimes n)a\]

        \item [(ii)] \[
            \mu(m \otimes n) = (mn,0) = (-nm,0) = -(-1)^{|m||n|}\mu(n \otimes n)
       \]
       \item[(iii)] If $u= (f,g) \in M _{\geq 1}^{\natural}$
       \[
           -(-1)^{|m|}\mu(m \otimes \mu(n \otimes u)) = -(mnf,0) = (-mnf,0) = \mu(\mu(m\otimes n)\otimes u)
       \]
       and otherwise if $u \in M_0$, then
       \[
       -(-1)^{|m|}\mu(m \otimes \mu(n \otimes u)) = -(mnu,0) = (-mnu,0) = \mu(\mu(m\otimes n)\otimes u).
       \]
       \item[(iv)] 
       \[
       \partial^M\mu(m \otimes n) = \partial^M(mn,0) = -\partial^B(mn) = m\partial^B(n)-\partial^B(m)n = m \beta(n)-\beta(m)n
       \]
    \end{itemize}

Thus, $\mu$ is an $A$-linear nullhomotopy for $\theta$ satisfying (\ref{twisting data}) and (\ref{gradedcommutative condition}), which is also easily seen to satisfy $\beta\mu =0$ by definition.

Finally, suppose that $B$ is strictly commutative and $m\in M^{\natural}$ with $|m|$ even. If $m=(b,c) \in M_{\geq 1}^{\natural}$, then $|b|$ is odd so that $\mu(m\otimes m) = (b^2,0)= (0,0)$, and otherwise if $m \in M_0 =B_1$, then $\mu(m\otimes m) = (m^2,0) = (0,0)$. Hence, $\mu$ satisfies (\ref{strictly commutative condition}).

\end{proof}

\section{Applications}\label{apps}

We now apply the construction in the previous section to two different families of ideals, adapting when necessary.  We only provide applications involving and obtaining strictly commutative dg algebras.  For brevity, we will simply write ``dg algebra" to mean strictly commutative dg algebra.    

\subsection{Complete Intersection-like Ideals}

We begin with an application related to resolutions of ideals built from complete intersections.  

\begin{cor}\label{CIdg}
    Let $(R,\mathfrak{m})$ be a Noetherian local ring or standard graded algebra over a field with irrelevant ideal $\mathfrak{m}$. Suppose that $C,L \subseteq R$ are complete intersection ideals with $C \subseteq \mathfrak{m}L$ and that $f\in R$ is regular on $R$ and $R/C$ (all homogeneous in the graded case).  Set $I = C +fL$. Then, the minimal free resolution of $R/I$ over $R$ has the structure of a dg algebra.
\end{cor}

\begin{proof}
    We construct the minimal free resolution of $R/I$ over $R$ via two mapping cones. Consider the short exact sequence
    \[
    0 \to I/C \to R/C \to R/I \to 0.
    \]
    Since $f$ is a nonzerodivisor on $R/C$, we have 
    \[
    L/C \cong (C+fL)/C = I/C
    \]
    via multiplication by $f$. Now, consider the short exact sequence
    \[
    0 \to C \to L \to L/C\to 0.
    \]
    Let $\KK^C$ and $\KK^L$ denote the Koszul complexes minimally resolving $R/C$ and $R/L$, respectively, where we view the underlying graded modules of the complexes as $\bigwedge R^b$ and $\bigwedge R^a$. If $C=(f_1, \ldots, f_b)$ and $L=(g_1, \ldots, g_a)$, then, since $C\subseteq \mathfrak{m}L$, we can write $f_j = \sum_{i=1}^{a}g_iy_{i,j}$ for some $a\times b$ matrix $Y= (y_{i,j})$ with entries in $\mathfrak{m}$ (homogeneous in the graded case). It is not too hard to verify that the $R$-algebra map $\alpha = \bigwedge Y : \KK^C \to \KK^L$ is a chain map lifting the natural surjection $R/C \to R/L$. Since $\KK^C$ and $\KK^L$ are dg algebras generated in degree one and $\alpha$ is an $R$-algebra map, it suffices to note that $\alpha$ is a chain map up to degree one. Hence, $\alpha$ is a homomorphism of dg algebras. In particular, we note that $M := \text{Cone}(\alpha_{\geq 1})[1]$ is the minimal free resolution of $L/C$.

    By \cref{thm-Betadef}, $M$ has the structure of a dg $\KK^C$-module with twisting data $\beta, \mu$. Hence, the minimal free resolution of $I/C$ is $M$, where we simply multiply the augmentation  map $M \to L/C$ by $f$ to obtain the augmentation map $M \to I/C$. Then, $f\beta: M \to \KK^C$ is a homomorphism of dg modules lifting the inclusion $I/C\subseteq R/C$ so that $C = \text{Cone}(f\beta)$ is the minimal free resolution of $R/I$.  It is easily seen that $f\beta, f\mu$ is also twisting data for $M$, so that $C$ has the structure of a dg algebra by \cref{thm: dgmodule M} and \cref{thm-Betadef}.
\end{proof}

An ideal $I$ of $Q=\Bbbk[x_1,\ldots,x_n]$ is \textbf{Koszul} if $\Bbbk$ has a linear minimal free resolution over $Q/I$.  Two ideals $I,J \subseteq Q$ are \textbf{linked} if there is a complete intersection ideal $C$ with $C \subseteq I \cap J$, $C:I=J$, and $C:J=I$.  If $I$ and $J$ are linked, we write $I \sim J$.  Call $I$ is \textbf{licci} if there is a sequence of links $I \sim J_1 \sim J_2 \sim \dots \sim J_s$ in which $J_s$ is a complete intersection.  In \cite{koszullicci}, the authors classify Koszul licci ideals $I$ into four types based on the deviation $\delta(I)$ of $I$: (1) when $\delta(I)=0$, (2) when $\delta(I)=1$, (3) when $\delta(I)=2$ and $I$ is Gorenstein, and (4) otherwise.  In case (4), it is shown that 
\[
I=wI_1(X)+I_1(XY)+\mathfrak{q},
\]
where $X$ and $Y$ are matrices of linear forms, $w$ is a linear form, $I_1(-)$ is the ideal generated by the entries of the input matrix, and $\mathfrak{q}$ is generated by a regular sequence on $\dfrac{Q}{wI_1(X)+I_1(XY)}$.  In cases (1), (2), and (3), the ideal $I$ is minimally resolved by a dg algebra: case (1) ideals are complete intersections, which are minimally resolved by the Koszul complex; case (2) ideals are Cohen-Macaulay of height 2 and thus minimally resolved by the Hilbert-Burch complex \cite{hilbert}, which is known to admit a dg algebra structure \cite{burch}; and case (3) ideals are Gorenstein of height 3, shown to be minimally resolved by a dg algebra in \cite{kustinmiller}.  Corollary \ref{CIdg} takes care of case (4) ideals.  

\begin{cor}\label{cor:koszullicci}
    Koszul licci ideals are minimally resolved by dg algebras.  
\end{cor}

\begin{proof}
    Both $I_1(X)$ and $I_1(XY)$, being generated by linear forms in $Q$, are complete intersection ideals.  Taking $L=I_1(X)$, $C=I_1(XY)$, and $f=w$ in Corollary \ref{CIdg}, we get that $wI_1(X)+I_1(XY)$ is minimally resolved by a dg algebra.  By tensoring the minimal free resolution of $wI_1(X)+I_1(XY)$ by the Koszul complex minimally resolving $\mathfrak{q}$, we obtain the minimal free resolution of $I$, since $\mathfrak{q}$ is generated by a regular sequence on $\dfrac{Q}{wI_1(X)+I_1(XY)}$.  The tensor product of dg algebras is again a dg algebra.      
\end{proof}

\subsection{Suspension of Edge Ideals}\label{suspensionsection}

We now construct the minimal free resolution of the edge ideal of a ``new" graph obtained via the introduction of a new vertex to an ``old" graph.  Let $G$ be a finite simple graph on the vertex set $V$, and let $I(G)$ be its edge ideal in $\Bbbk[V]$.  Let $\mathcal{C}$ be a vertex cover of $G$, and let $z$ be a variable not in $V$.  The \textbf{suspension} of $z$ over $\mathcal{C}$ is the graph $G^\mathcal{C}$ on the vertex set $V \cup\{z\}$ with edge set $E(G) \cup \{zx \mid x \in \mathcal{C}\}$.  The edge ideal $I(G^\mathcal{C})$ of $G^\mathcal{C}$ then lives in the polynomial ring $Q=\Bbbk[V \cup \{z\}]$.  

\begin{ex}
    Figure \ref{suspfig} shows a graph $G$ on the vertices $x_1,\ldots,x_5$ and the suspension of a new vertex $z$ over $\mathcal{C}=\{x_1,x_3,x_4\}$, a vertex cover of $G$.
\end{ex}
    \begin{figure}[h]
        \centering
\begin{tikzpicture}[
    scale=0.85,
    every node/.style={font=\small},
    vertex/.style={circle, draw, fill=white, minimum size=5mm, inner sep=0pt},
    edge/.style={thick},
    >={Latex}
]

\begin{scope}
    \node[vertex] (v4) at (-3,0) {$x_4$};
    \node[vertex] (v2) at (-1,0) {$x_2$};
    \node[vertex] (v3) at (1,0) {$x_3$};
    \node[vertex] (v5) at (3,0) {$x_5$};
    \node[vertex] (v1) at (0,1.4) {$x_1$};

    \draw[edge] (v1) -- (v3);
    \draw[edge] (v2) -- (v3);
    \draw[edge] (v3) -- (v5);
    \draw[edge] (v2) -- (v4);
    \draw[edge] (v2) -- (v1);

    \node[align=center] at (0,-1)
    {$I(G)=(x_1x_2,\ x_1x_3,\ x_2x_3,\ x_2x_4,\ x_3x_5)$};
\end{scope}

\draw[->, thick] (3.85,0.75) -- (5.25,0.75);

\begin{scope}[xshift=9cm]
    \node[vertex] (v4) at (-3,0) {$x_4$};
    \node[vertex] (v2) at (-1,0) {$x_2$};
    \node[vertex] (v3) at (1,0) {$x_3$};
    \node[vertex] (v5) at (3,0) {$x_5$};
    \node[vertex] (v1) at (0,1.4) {$x_1$};
    \node[vertex] (z)  at (0,3) {$z$};

    \draw[edge] (v1) -- (v3);
    \draw[edge] (v2) -- (v3);
    \draw[edge] (v3) -- (v5);
    \draw[edge] (v2) -- (v4);
    \draw[edge] (v2) -- (v1);
    \draw[edge] (z) -- (v1);
    \draw[edge] (z) -- (v3);
    \draw[edge] (z) -- (v4);

    \node[align=center] at (0,-1)
    {$I(G^{\mathcal{C}})=(x_1x_2,\ x_1x_3,\ x_2x_3,\ x_2x_4,\ x_3x_5,\ zx_1,\ zx_3,\ zx_4)$};
\end{scope}

\end{tikzpicture}
        \caption{Suspension over a vertex cover}
        \label{suspfig}
    \end{figure}
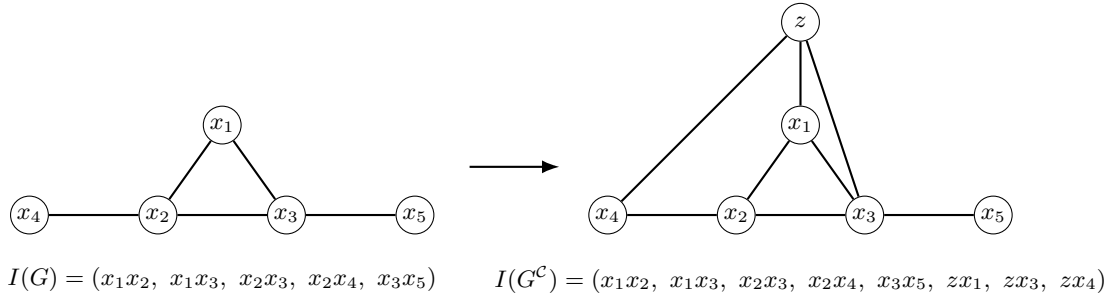

We will look at the case when the minimal free resolution of $Q/I(G)$ admits the structure of a dg algebra.  For a squarefree monomial ideal $I$ (such as an edge ideal), Katth\"an shows that if the minimal free resolution of $Q/I$ is a dg algebra, then it must have a certain form \cite[Theorem 3.6]{Kat}.

\begin{thm}[Katth\"an, 2019]\label{structurethm}
Let $I$ be a squarefree monomial ideal, and suppose that the minimal $Q$-free resolution $\mathbb{G}$ of $Q/I$ admits the structure of a differential graded algebra.  Then, there is an isomorphism of differential graded algebras $\mathbb{G} \cong \mathbb{T}/\mathbb{J}$, where $\mathbb{T}$ is the Taylor resolution of $Q/I$ and $\mathbb{J}$ is a dg ideal of $\mathbb{T}$.  
\end{thm}

Throughout, let $\mathbb{F}$ be the Taylor resolution on $\la zx \mid x \in \mathcal{C} \ra$, $\mathbb{G}$ be the minimal free resolution of $Q/I(G)$, and $\mathbb{K}$ be the Koszul complex on $\mathcal{C}$, i.e., resolving the quotient by $\la x \mid x \in \mathcal{C}\ra$.  For $W$ a subset of the relevant generating set, the basis elements of $\mathbb{F}$ will be denoted $f_W$ and the basis elements of $\mathbb{K}$ will be denoted $e_W$.  

Suppose $\mathbb{G}$ is a dg algebra.  By Theorem \ref{structurethm}, there is a surjection $\pi:\mathbb{T} \onto \mathbb{G}$ whose kernel $\mathbb{J}$ is a dg ideal of $\mathbb{T}$, where $\mathbb{T}$ is the Taylor resolution on the edges of $G$ \cite{Kat}.  Viewing $\mathbb{G}$ as this quotient, we denote basis elements of $\mathbb{G}$ by $g_W$, where $W$ is a subset of the generators of $I(G)$.  We will likewise write $g_W$ for the generators of the Taylor resolution $\mathbb{T}$ of which $\GG$ is a quotient.      


Moving forward, we will abuse notation by writing the edge $\{x_i,x_j\} \in E(G)$ as the quadratic $x_ix_j$. This simplifies the notation in our following definitions.
\begin{defn}
    Let $G$ be a graph and $\cC$ a vertex cover of $G$. We call $\fX : E(G) \to  V(G)$ a \textbf{cover map (of $(G,\cC)$)} if $\fX(w) \in \cC$ for all $w \in E(G)$. We say $\fY : E(G) \to V(G)$ is the \textbf{co-cover map (with respect to $\fX$)} if $\fY(w) = w / \fX(w)$ for all $w \in E(G)$.
\end{defn}

To match the prior work of the first, second, and fourth authors in \cite{GMP}, we introduce the following notation.
\begin{defn}\label{xWyW}
    Consider a graph $G$, and fix a vertex cover $\cC$ and maps $\fX$ and $\fY$ as defined above. For every subset $W \subseteq E(G)$, we set $x_W := \lcm \{\fX(w) : w \in W\}$ and $y_W := \lcm \{\fY(w) : w \in W \}$. Further, we set $d_W := \gcd \{x_W, y_W\}$.
\end{defn}

This notation appears in \cite[Section 4]{GMP} as part of an explicit construction of the minimal free resolution of the edge ideal of a diameter-four tree. A key difference is that the prior work did not consider $d_W$. In fact, translating this notation to the prior work, one would find $d_W = 1$, since in that case no two vertices in the vertex cover are adjacent to each other. In other words, for any edge $w$, there was only one choice for $\fX(w)$.

This changes in our current context since we are considering any finite, simple graph $G$ and any vertex cover $\cC$. Thus, for some edge $w = x_i x_j \in E(G)$ it is possible that $x_i, x_j \in \cC$, and so one can choose $\fX(w) = x_i$ or $\fX(w) = x_j$. If one chooses $\fX(w) = x_i$, then $\fY(w) = x_j$. Now suppose there exists another vertex $x_{\ell} \in V(G)$ such that $x_{\ell} \notin \cC$ but $w' = x_j x_{\ell} \in E(G)$, then $\fX(w') = x_j$. As a consequence, if $w, w' \in W \subseteq E(G)$, then $x_j$ divides both $x_W$ and $y_W$. This means $x_j$ divides $d_W$, hence $d_W \neq 1$. Since $x_W$ and $y_W$ are not necessarily coprime, we note the following lemma.
\begin{lem}\label{lcm*gcd}
    Consider the set-up of Definition \ref{xWyW}. If $\emptyset \neq W \subseteq E(G)$, then $x_W y_W = d_W \lcm W$. 
\end{lem}

\begin{proof}
    Note that for any $w \in W$, we have $w = \fX(w) \fY(w)$ which allows us to observe
    \begin{align*}
        \lcm W &= \lcm \{\fX(w)\fY(w) : w \in W \} \\
         &= \lcm \left\{\lcm \{\fX(w) : w \in W \}, \lcm \{\fY(w) : w \in W \} \right\} \\
         &= \lcm \{x_W, y_W\}.
    \end{align*}
    The result then follows from the fact that $x_W y_W = \gcd \{x_W, y_W\} \lcm \{x_W, y_W\}$.
\end{proof}

As an immediate consequence, we can prove the following.
\begin{lem}
\label{lem:Lcm}
    Consider the set-up of Definition \ref{xWyW}, and suppose $W, V \subseteq E(G)$. Set $\fX(W) := \{\fX(w): w \in W\}$. If $\fX(W) \cap \fX(V) = \emptyset$, then \[\frac{\lcm W \lcm V}{\lcm (W \cup V)} = \frac{y_W y_V d_{W \cup V} }{d_W d_V y_{W \cup V}}.\]
\end{lem}

\begin{proof}
    Observe that
\begin{align*}
    \frac{\lcm W \lcm V}{\lcm (W \cup V)} &= \frac{\left(\frac{x_Wy_W}{d_W}\right) \left(\frac{x_Vy_V}{d_V}\right)}{\left(\frac{x_{W\cup V} y_{W \cup V}}{d_{W \cup V}}\right)} \\
     &= \frac{x_W x_V}{x_{W \cup V}} \frac{y_W y_V d_{W \cup V} }{d_W d_V y_{W \cup V}}.
\end{align*}
Since $\fX(W) \cap \fX(V) = \emptyset$, we have $x_W x_V = x_{W \cup V}$, and thus the above reduces to the desired equation.
\end{proof}


To set up the following lemmas, fix a total order $<_1$ on the elements of $\cC$. Let $<_2$ be a total order on $E(G)$ such that for any $w, w' \in E(G)$,  if $\fX(w) <_1 \fX(w')$, then $w <_2 w'$. Since it will be clear from the context, we will denoted both of these total orders by $<$.

\begin{lem}\label{phi}
    Define $\varphi:\mathbb{T} \to \mathbb{K}$ by sending 1 to 1 and each $g_W$ to $\frac{y_W}{d_W}e_{W'}$, where $W' = \mathfrak{X}(W)$ if $|W|=|\mathfrak{X}(W)|$ and $e_{W'}=0$ if not.  Then, $\varphi$ is a chain map of dg algebras. 
\end{lem}

\begin{proof}
    Since $I(G) \subseteq \la x \mid x \in \mathcal{C} \ra$, there is a surjection $Q/I(G) \onto Q/\la x \mid x \in \mathcal{C} \ra$.  It is straightforward to see that
   \[
\begin{tikzcd}
\mathbb{T}_0 \arrow[r, two heads] \arrow[d, "\text{id}"] & Q/I(G) \arrow[d, two heads] \\
\mathbb{K}_0 \arrow[r, two heads]                        & Q/\langle x \mid x \in \mathcal{C} \rangle        
\end{tikzcd}
    \]
    commutes.  We now look at higher-degree cases.  

    We begin with the case that $e_{W'}\neq 0$.  In this case, 
    \begin{align*}
        \partial^\mathbb{K}(\varphi(g_W))&=\partial^\mathbb{K}\left(\frac{y_W}{d_W}e_{W'}\right)\\
        &=\frac{y_W}{d_W}\partial^\mathbb{K}(e_{W'})\\
        &=\frac{y_W}{d_W}\sum_{\mathfrak{X}(w) \in W'} (-1)^{\sigma(\fX(w),W')} \fX(w) e_{W'\smallsetminus\{\fX(w)\}}.
    \end{align*}
    Since $e_{W'} \neq 0$ implies $e_{(W \smallsetminus \{w\})'} = e_{W' \smallsetminus \{\fX(w)\}}$, we also find
    \begin{align*}
        \varphi(\partial^\mathbb{T}(g_W))&=\varphi\left(\sum_{w \in W}(-1)^{\sigma(w,W)}\dfrac{\lcm W}{\lcm(W\smallsetminus\{w\})}g_{W\smallsetminus\{w\}}\right)\\
        &= \sum_{w \in W}(-1)^{\sigma(w,W)}\dfrac{x_W y_W d_{W \smallsetminus\{w\}}}{x_{W\smallsetminus\{w\}} y_{W\smallsetminus\{w\}} d_W} \left(\dfrac{y_{W \smallsetminus \{w\}}}{d_{W \smallsetminus \{w\}}} e_{(W\smallsetminus\{w\})'}\right) \\
        &= \sum_{w \in W}(-1)^{\sigma(w,W)}\dfrac{x_W y_W}{x_{W\smallsetminus\{w\}} d_W} e_{W'\smallsetminus\{\fX(w)\}} \\
        &= \dfrac{y_W}{d_W} \sum_{w \in W}(-1)^{\sigma(w,W)} \fX(w) e_{W'\smallsetminus\{\fX(w)\}} \\
        &=\frac{y_W}{d_W}\sum_{\mathfrak{X}(w) \in W'}(-1)^{\sigma(\fX(w),W')}\fX(w) e_{W'\smallsetminus\{\fX(w)\}},
    \end{align*}
where $\sigma(w,W)$ is the number of elements of $W$ which $w$ is greater than under $<$.  

In the case that $e_{W'}=0$, there are distinct $w_1,w_2 \in W$ such that $\fX(w_1)=\fX(w_2)$.  Without loss of generality, $\sigma(w_1,W)=\sigma(w_2,W)+1$.  Then,
\begin{align*}
    \partial^\mathbb{T}(g_W)&=(-1)^{\sigma(w_1,W)}\dfrac{\lcm W}{\lcm(W \smallsetminus\{w_1\})}g_{W \smallsetminus\{w_1\}}-(-1)^{\sigma(w_1,W)}\dfrac{\lcm W}{\lcm(W \smallsetminus\{w_2\})}g_{W \smallsetminus\{w_2\}}\\
    &+\sum_{w \in W\smallsetminus\{w_1,w_2\}}(-1)^{\sigma(w,W)}\dfrac{\lcm W}{\lcm(W\smallsetminus \{w\})}g_{W\smallsetminus \{w\}}.
\end{align*}
Since $w_1,w_2 \in W \smallsetminus\{w\}$ for all $w\in W\smallsetminus\{w_1,w_2\}$, the summation 
\[
\sum_{w \in W\smallsetminus\{w_1,w_2\}}(-1)^{\sigma(w,W)}\dfrac{\lcm W}{\lcm(W\smallsetminus \{w\})}g_{W\smallsetminus \{w\}}
\]
goes to zero under $\varphi$.  Now, since $\fX(w_1) = \fX(w_2)$, we have $\fX(W \smallsetminus \{w_1\}) = \fX(W) = \fX(W \smallsetminus \{w_2\})$. It follows that $e_{(W \smallsetminus \{w_1\})'} = e_{(W \smallsetminus \{w_2\})'}$ and

\begin{align*}
    (-1)^{\sigma(w_1,W)}\varphi\left(\dfrac{\lcm W}{\lcm(W \smallsetminus\{w_1\})}g_{W \smallsetminus\{w_1\}} - \dfrac{\lcm W}{\lcm(W \smallsetminus\{w_2\})}g_{W \smallsetminus\{w_2\}}\right) \hspace{-10cm} &  \\
    & = (-1)^{\sigma(w_1,W)} \left(\dfrac{\lcm W}{\lcm(W \smallsetminus\{w_1\})}\frac{y_{W \smallsetminus\{w_1\}}}{d_{{W \smallsetminus\{w_1\}}}}e_{(W \smallsetminus\{w_1\})'} - \dfrac{\lcm W}{\lcm(W \smallsetminus\{w_2\})}\frac{y_{W \smallsetminus\{w_2\}}}{d_{{W \smallsetminus\{w_2\}}}}e_{(W \smallsetminus\{w_2\})'} \right) \\
    & = (-1)^{\sigma(w_1,W)}\left(\dfrac{x_Wy_W}{d_Wx_{W \smallsetminus\{w_1\}}}e_{(W \smallsetminus\{w_1\})'} - \dfrac{x_Wy_W}{d_Wx_{W \smallsetminus\{w_2\}}}e_{(W \smallsetminus\{w_2\})'}\right) \\
    & = (-1)^{\sigma(w_1,W)}\left(\dfrac{y_W}{d_W} \fX(w_1) e_{(W \smallsetminus\{w_1\})'} - \dfrac{y_W}{d_W} \fX(w_2) e_{(W \smallsetminus\{w_2\})'}\right) \\
    & = (-1)^{\sigma(w_1,W)}\left(\dfrac{y_W}{d_W} \fX(w_1) e_{(W \smallsetminus\{w_1\})'} - \dfrac{y_W}{d_W} \fX(w_1) e_{(W \smallsetminus\{w_1\})'}\right) \\
    &=0.
\end{align*}
 
Since $\partial^\mathbb{K}(\varphi(g_W))=0$, as well, the proof that $\varphi$ is a chain map is complete.

To see that $\varphi$ is a map of dg algebras, we begin by examining $\varphi(g_Vg_W)$.  If $V\cap W=\emptyset$, then
\begin{align*}
    \varphi(g_Vg_W)&=\varphi\left((-1)^{\sigma(V,W)}\dfrac{\lcm V\lcm W}{\lcm (V \cup W)}g_{V \cup W}\right) \\
    & = (-1)^{\sigma(V,W)}\dfrac{\lcm V\lcm W}{\lcm (V \cup W)}\varphi\left(g_{V \cup W}\right) \\
    & = (-1)^{\sigma(V,W)}\dfrac{\lcm V\lcm W}{\lcm (V \cup W)} \frac{y_{V \cup W}}{d_{V \cup W}} e_{(V \cup W)'},
\end{align*}
where $e_{(V \cup W)'} \neq 0$ if and only if $|\fX(V)| = |V|$, $|\fX(W)| = |W|$, and $\fX(V) \cap \fX(W) = \emptyset$.  Here, $\sigma(V,W)$ is the number of elements $(v,w)$ of $V \times W$  with $v>w$.
It follows that either $e_{(V \cup W)'} = 0$ or we can apply Lemma \ref{lem:Lcm} to get 
\[\varphi(g_Vg_W) = (-1)^{\sigma(V,W)} \dfrac{y_V}{d_V}\dfrac{y_W}{d_W} e_{(V \cup W)'}.\]

By definition of $\varphi$, we have $\varphi(g_V)\varphi(g_W)=\frac{y_V}{d_V}\frac{y_W}{d_W}e_{V'}e_{W'}$, which is equal to $\varphi(g_V g_W)$ if and only if $e_{V'} e_{W'} = (-1)^{\sigma(V,W)} e_{(V \cup W)'}$. If $e_{V'} = 0$, then, by definition, we must have $|\fX(V)| \neq |V|$. In turn, this means $|\fX(V \cup W)| \neq |V \cup W|$ and $e_{(V \cup W)'} = 0$. As such, we have
\[e_{V'} e_{W'} = 0 \cdot e_{W'} = 0 = (-1)^{\sigma(V,W)} \cdot 0 = (-1)^{\sigma(V,W)} e_{(V \cup W)'}.\]
A similar argument holds for when $e_{W'} = 0$.

We may now assume that $e_{V'} \neq 0 \neq e_{W'}$ meaning $e_{V'} = e_{\fX(V)}$ and $e_{W'} = e_{\fX(W)}$. If $\fX(V) \cap \fX(W) \neq \emptyset$, then the Koszul product gives us $e_{\fX{V}} e_{\fX{W}} = 0$. Moreover, we have $|\fX(V \cup W)| = |\fX(V) \cup \fX(W)| < |V \cup W|$ and hence
\[(-1)^{\sigma(V,W)} e_{(V \cup W)'} = (-1)^{\sigma(V,W)} \cdot 0 = 0 = e_{\fX(V)} e_{\fX(W)} = e_{V'} e_{W'}.\]

Finally, if $\fX(V) \cap \fX(W) = \emptyset$, then we have $|\fX(V \cup W)| = |V \cup W|$, and thus $e_{(V \cup W)'} = e_{\fX(V \cup W)}$. Moreover, the correspondence on the total order on $E(G)$ with the total order on $\cC$ gives $\sigma(V,W) = \sigma(\fX(V),\fX(W))$. Hence, the Koszul product produces
\begin{align*}
    e_{V'} e_{W'} & = e_{\fX(V)} e_{\fX(W)} \\
     & = (-1)^{\sigma(\fX(V),\fX(W))} e_{\fX(V) \cup \fX(W)} \\
     & = (-1)^{\sigma(\fX(V),\fX(W))} e_{\fX(V \cup W)} \\
     & = (-1)^{\sigma(V,W)} e_{(V \cup W)'}.
\end{align*}
Thus, it holds that $e_{V'} e_{W'} = (-1)^{\sigma(V,W)} e_{(V \cup W)'}$ and implies $\varphi(g_V) \varphi(g_W) = \varphi(g_V g_W)$ when $V \cap W = \emptyset$.

On the other hand, if $V \cap W \neq \emptyset$, then $\varphi(g_Vg_W)=0=\varphi(g_V)\varphi(g_W)$.  Thus, $\varphi$ is a map of dg algebras.   
\end{proof}

Standard homological algebra tells us that there is a chain map $\tau:\mathbb{G} \to \mathbb{K}$ lifting the surjection $Q/I \onto Q/\la x \mid x \in \mathcal{C} \ra$.  The next lemma says that we may take $\tau$ as we took $\varphi$.   

\begin{lem}\label{tau}
    The map $\tau:\mathbb{G} \to \mathbb{K}$ may be taken to be given by $\tau(1)=1$ and $\tau(g_W)=\frac{y_W}{d_W}e_{W'}$, where the notation is as in Lemma \ref{phi}.  
\end{lem}

\begin{proof}
    The proof the $\tau$ is a chain map lifting the surjection $Q/I \onto Q/\la x \mid x \in \mathcal{C} \ra$ is essentially the same as the analog for $\varphi$, found in the proof of Lemma \ref{phi}.  
\end{proof}

Henceforth, $\tau$ will be assumed to have the definition in Lemma \ref{tau}.

\begin{lem}
\label{Lem:tau-dgcondition}
    The map $\tau$ is a dg morphism if and only if $\mathbb{J} \subseteq \ker \varphi$.
\end{lem}

\begin{proof}
    The proof here is just noting that the given conditions are equivalent to the triangle
    \[
\begin{tikzcd}
\mathbb{T} \arrow[d, "\pi"'] \arrow[rd, "\varphi"] &            \\
\mathbb{G} \arrow[r, "\tau"']                      & \mathbb{K}
\end{tikzcd}
    \]
    commuting.  
\end{proof}

We now use $\tau$ to build a map $\alpha$, the mapping cone of which will be the a key ingredient in the construction of the minimal free resolution of $Q/I(G^\mathcal{C})$.

\begin{lem}\label{mulem}
    Define $\mu^z:\mathbb{K} \to \mathbb{F}$ by sending $1$ to $z$ and $e_{x_1,\ldots,x_q}$ to $f_{zx_1,\ldots,zx_q}$.  The map
    \[
    \alpha:=\mu^z \circ \tau:\mathbb{G} \to \mathbb{F}
    \]
    is a chain map, but it is not a map of dg algebras.   
\end{lem}

\begin{proof}
Note that $z=\mu^z(1)=\mu^z(1 \cdot 1)$, but $\mu^z(1)\cdot \mu^z(1) =z^2$, so that $\mu^z$, and thus $\alpha$, is not a map of dg algebras.  

    By Lemma \ref{tau}, $\tau$ is a chain map.  So, it remains to show that $\mu^z$ is a chain map.  Indeed,
     \[
\begin{tikzcd}
\mathbb{K}_1 \arrow[r,"\partial^{\KK}_1"] \arrow[d, "\mu^z_1"] & \mathbb{K}_0 \arrow[d,"\cdot z"] \\
\mathbb{F}_1 \arrow[r,"\partial^{\FF}_1"]                        & \mathbb{F}_0
\end{tikzcd}
    \]
    commutes because $\partial_1^\FF(f_{zx})=zx=z \cdot \partial^\KK_1(e_x)$.  In higher degrees
    \begin{align*}
        \mu^z\left(\partial^\KK(e_W)\right)&=\mu^z\left(\sum_{i=1}^{|W|} (-1)^{i-1}x_ie_{W\smallsetminus\{x_i\}}\right)\\
        &=\sum_{i=1}^{|W|} (-1)^{i-1}x_if_{W_z\smallsetminus\{zx_i\}},
    \end{align*}
    where $W=\{x_1<\ldots<x_{|W|}\}$ and $W_z=\{zx_1<\ldots<zx_{|W|}\}$.  Furthermore,
    \begin{align*}
        \partial^\FF \left(\mu^z(e_W)\right)&=\partial^\FF \left(f_{W_z}\right)\\
        &=\sum_{i=1}^{|W|} (-1)^{i-1}x_if_{W_z\smallsetminus\{zx_i\}},
    \end{align*}
    completing the proof.  
\end{proof}

Regardless of the fact that $\alpha$ is not a dg morphism, we can still recover a dg module over $\mathbb{G}$.

\begin{lem}\label{newmodstructure}
    If $\tau$ is a dg morphism, then the complex $M:=\Cone{\alpha_{\geq 1}}[1]$ is a dg module over $\mathbb{G}$.
\end{lem}

\begin{proof}
  By definition, $\alpha = \mu^z \circ \tau$.  Since $\tau$ is a dg morphism and $\mu^z(ab) = \frac{1}{z}\mu^z(a)\mu^z(b)$ for all $a,b \in \KK$, we have
    \[
    \alpha(ab) = \tfrac{1}{z}\alpha(a)\alpha(b)
    \]
     for all $a,b \in \GG$.  Changing the action in Theorem \ref{thm-Betadef} (a) to 
    \[
    am =
    \begin{cases}
    ((-1)^{|a|}\frac{1}{z}\alpha(a)b,ac), & m=(b,c)\in M_{\geq 1}^{\natural}\\
    ((-1)^{|a|}\frac{1}{z}\alpha(a)m, -\frac{1}{z}a\partial^B(m)), & m \in M_0=\FF_1, a \in \GG_{\geq 1}^{\natural},
    \end{cases}
    \]
    a similar proof to that of Theorem \ref{thm-Betadef} (a) shows that $\Cone{\alpha_{\geq 1}}[1]$ is a dg module over $\GG$.
\end{proof}

For a graph $G$ and a vertex cover $\mathcal{C}$ of $G$, define the directed graph $D(G,\mathcal{C})$ to have vertex set the vertex set of $G$ and directed edge $(a,b)$ if and only if $\{a,b\}$ is an edge in $G$ and $b \in \mathfrak{X}(E(G))$.  By definition of $\varphi$ (Lemma \ref{phi}), and thus $\tau$ (Lemma \ref{tau}), we observe that $\varphi$, and thus $\tau$, will produce a unit coefficient on an output (in positive degrees) precisely when $y_W = d_W$ for some $W \subseteq E(G)$.  The only way to get $y_W=d_W$ is if $W$ is an oriented cycle in $D(G,\mathcal{C})$.  Thus, we have the following result.

\begin{lem}\label{mintau}
    The map $\tau$ satisfies $\tau_{\geq 1} \otimes \Bbbk=0$ if and only if $D(G,\mathcal{C})$ has no oriented cycles.
\end{lem}

\begin{lem}\label{middleconeres}
    The complex $\Cone{\alpha_{\geq 1}}[1]$ resolves $\dfrac{\la zx \mid x \in \mathcal{C}\ra+zI(G)}{zI(G)}$ over $Q$.  If $D(G,\mathcal{C})$ has no oriented cycles, then $\Cone{\alpha_{\geq 1}}[1]$ is minimal.
\end{lem}

\begin{proof}
    Since $\GG$ and $\FF$ are resolutions, the long exact sequence in homology associated to $\Cone{\alpha_{\geq 1}}$ reduces to
    \[
    0 \xrightarrow[]{} H_2(\Cone{\alpha_{\geq 1}}) \xrightarrow[]{} H_1(\GG_{\geq 1}) \xrightarrow[]{} H_1(\FF_{\geq 1}) \xrightarrow[]{} H_1(\Cone{\alpha_{\geq 1}}) \xrightarrow[]{} 0,
    \]
    which, since $\GG$ resolves $Q/I(G)$ and $\FF$ resolves $Q/\la zx \mid x \in \mathcal{C}\ra$, may be written
    \[
    0 \xrightarrow[]{} H_2(\Cone{\alpha_{\geq 1}}) \xrightarrow[]{} I(G) \xrightarrow[]{\cdot z} \la zx \mid x \in \mathcal{C}\ra \xrightarrow[]{} H_1(\Cone{\alpha_{\geq 1}}) \xrightarrow[]{} 0.
    \]
    Since $z$ is a nonzerodivisor on $Q$, we see that $H_2(\Cone{\alpha_{\geq 1}})=0$ and
    \begin{align*}
    H_0(\Cone{\alpha_{\geq 1}}[1])&=H_1(\Cone{\alpha_{\geq 1}})\\
    &=\dfrac{\la zx \mid x \in \mathcal{C}\ra}{zI(G)}\\
    &=\dfrac{\la zx \mid x \in \mathcal{C}\ra}{\la zx \mid x \in \mathcal{C}\ra \cap zI(G)}\\
    &\cong \dfrac{\la zx \mid x \in \mathcal{C}\ra+zI(G)}{zI(G)},
    \end{align*}
    and we see that $\Cone{\alpha_{\geq 1}}[1]$ resolves what we claimed it does.  Both $\GG$ and $\FF$ are minimal, and $\alpha_{\geq 1} \otimes \Bbbk=0$ if $\tau$ produces coefficients in positive homological degrees, which is the case if $D(G,\mathcal{C})$ has no oriented cycles (Lemma \ref{mintau}).  Thus, in this case, $\Cone{\alpha_{\geq 1}}[1]$ is minimal.  
\end{proof}

\begin{lem}\label{isos}
    There is an isomorphism
    \[
    \dfrac{\la zx \mid x \in \mathcal{C}\ra+zI(G)}{zI(G)} \cong \dfrac{\la zx \mid x \in \mathcal{C}\ra+I(G)}{I(G)}.
    \]
\end{lem}

\begin{proof}
    This will follow from the Second Isomorphism Theorem if we can show that
    \[
    \la zx \mid x \in \mathcal{C}\ra \cap\, zI(G) = \la zx \mid x \in \mathcal{C}\ra \cap I(G).
    \]
    That the left-hand side is contained in the right-hand side is clear.  For the other containment, we note that
    \[
    \la zx \mid x \in \mathcal{C}\ra \cap I(G) \subseteq (z) \cap I(G)=zI(G),
    \]
    and so
    \[
    \la zx \mid x \in \mathcal{C}\ra \cap I(G)=\la zx \mid x \in \mathcal{C}\ra \cap\la zx \mid x \in \mathcal{C}\ra \cap I(G)\subseteq \la zx \mid x \in \mathcal{C}\ra \cap zI(G).
    \]
\end{proof}

We are now ready to construct the minimal free resolution of $Q/I(G^\mathcal{C})$.  

\begin{thm}\label{suspensionres}
    Let $M=\Cone{\alpha_{\geq 1}}[1]$.  Define $\bar\beta:M \to \mathbb{G}$ by
    \[
    \bar\beta(m) = \begin{cases} z\beta(m), & m= (b,c) \in M_{\geq 1}^{\natural}\\
        \beta(m), & m \in M_0 = \FF_1,
        \end{cases}
    \]
    where $\beta$ is as in Theorem \ref{thm-Betadef}, with $A=\GG$ and $B=\FF$.  When $D(G,\mathcal{C})$ does not have any oriented cycles, $C:=\Cone{\bar\beta}$ is the minimal free resolution of $Q/I(G^\mathcal{C})$.  Furthermore, if $\tau$ is a dg morphism, then $\Cone{\bar\beta}$ is a dg algebra.   
\end{thm}

\begin{proof}
    By Lemma \ref{newmodstructure}, if $\tau$ is a dg morphism, then $\Cone{\alpha_{\geq 1}}[1]$ is a dg $\GG$-module.  We claim that remaining the hypotheses of Theorem \ref{thm: dgmodule M} are satisfied when the action of $A:=\mathbb{G}$ on $M$ is as described in Lemma \ref{newmodstructure}, where $\beta$ is replaced with $\bar\beta$.  Define $\mu:M \otimes M \to M$ as in Theorem \ref{thm-Betadef}, part (c).  Several things need to be checked, but only one is worth writing, as the others' proofs are, mutatis mutandis, in the proof of Theorem \ref{thm-Betadef}:
    \begin{itemize}
        \item[(I)] $\bar\beta$ is a chain map which respects the action of $A$ of $M$
        \item[(II)] $\mu$ is $A$-linear
        \item[(III)] $\bar\beta\mu=0$
        \item[(IV)] $\theta:M \otimes M \to M$, given by $\theta(m \otimes n)=m\bar\beta(n)-\bar\beta(m)n$, is nullhomotopic via $\mu$
        \item[(V)] $\mu$ satisfies equation \ref{twisting data}
        \item[(VI)] $\mu$ satisfies equation \ref{gradedcommutative condition}
        \item[(VII)] $\mu$ satisfies equation \ref{strictly commutative condition}
    \end{itemize}

\noindent (I) The map $\bar\beta: M \to A$ is $Q$-linear because $\beta$ is. We check simultaneously that $\bar\beta$ is a chain map and $A$-linear. If $m = (b,c) \in M_{\geq 1}^{\natural}$ and $a\in A_{\geq 1}^{\natural}$, then
\begin{align*}
    \bar\beta\partial^M(m) = \bar\beta(-\partial^B(b)-\alpha(c), \partial^A(c))= -z\partial^A(c) = -\partial^A(zc)=\partial^A\bar\beta(m)
\end{align*}
\begin{align*}
    \bar\beta(am) = \bar\beta((-1)^{|a|}\tfrac{1}{z}\alpha(a)b,ac) = -zac = a\bar\beta(m). 
\end{align*}
On the other hand, if $m \in M_0$ and $a \in A_{\geq 1}^{\natural}$, then
\[
\bar\beta\partial^M(m)=0=\partial^A\bar\beta(m)
\]
since $\bar\beta(m)\in Q=A_0$, and 
\[
\bar\beta(am)=\bar\beta((-1)^{|a|}\tfrac{1}{z}\alpha(a)m, -\tfrac{1}{z}a\partial^B(m)) = a\partial^B(m)=a\bar\beta(m).
\]
Hence, $\bar\beta$ is a homomorphism of dg $A$-modules.\\

\noindent (II)-(VII) The proofs of go as their counterparts do in Theorem \ref{thm-Betadef}.  In (II), one simply needs to include a $\tfrac{1}{z}$ when the action of $A$ on $M$ is performed, and in (IV), one notes an occasional instance of $z \cdot \tfrac{1}{z}=1$.

    Since $A$ and $M$ are both minimal and $\bar\beta \otimes \Bbbk = 0$, we have that $C$ is minimal.  The long exact sequence in homology associated to $\Cone{\bar\beta}$ reduces to
    \[
    0 \xrightarrow[]{} H_1(\Cone{\bar\beta}) \xrightarrow[]{} H_0(\Cone{\alpha_{\geq 1}}[1]) \xrightarrow[]{} H_0(\GG) \xrightarrow[]{} H_0(\Cone{\bar\beta}) \xrightarrow[]{} 0,
    \]
    which, by Lemma \ref{middleconeres} and the fact that $\GG$ resolves $Q/I(G)$, becomes
    \[
    0 \xrightarrow[]{} H_1(\Cone{\bar\beta}) \xrightarrow[]{} \dfrac{\la zx \mid x \in \mathcal{C}\ra+zI(G)}{zI(G)} \xrightarrow[]{} Q/I(G) \xrightarrow[]{} H_0(\Cone{\bar\beta}) \xrightarrow[]{} 0.
    \]
    
    We have
    \begin{align*}
    \dfrac{Q/I(G)}{\left(\la zx \mid x \in \mathcal{C}\ra+zI(G)\right)/\,zI(G)}
    &=\dfrac{Q/I(G)}{\left(\la zx \mid x \in \mathcal{C}\ra+I(G)\right)/I(G)}\\ &\cong \dfrac{Q}{\la zx \mid x \in \mathcal{C} \ra+I(G)}\\
    &=Q/I(G^\mathcal{C}).
    \end{align*}

    Furthermore,
    \begin{align*}
        H_0(\Cone{\bar\beta})&=\coker \partial_1^{\Cone{\bar\beta}}\\
        &=\dfrac{Q}{\im \partial_1^\mathbb{G} + \im \partial_1^\mathbb{F}}\\
        &=Q/I(G^\mathcal{C}).
    \end{align*}
    It follows that $H_1(\Cone{\bar\beta})=0$.  
    \end{proof}

    The case of the following corollary in which $I(G)$ is minimally resolved by a dg algebra follows quickly from the construction in Theorem \ref{suspensionres}.  See \cite{karavien} for similar results, obtained primarily via analysis of independence complexes, rather than of resolutions, and without reference to dg algebras.  

\begin{cor}\label{splitting}
    With $G$, $\mathcal{C}$, and $G^\mathcal{C}$ as before, we have
    \begin{itemize}
        \item[\textup{a)}] $\beta_{i,j}(I(G^\mathcal{C}))=\beta_{i,j}(I(G))+\beta_{i,j}(\la zx \mid x \in \mathcal{C}\ra)+\beta_{i-1,j-1}(I(G))$ 
        \item[\textup{b)}] $\reg I(G^\mathcal{C})=\reg I(G)$
        \item[\textup{c)}] $\pdim I(G^\mathcal{C})=\max\{|\mathcal{C}|-1,\pdim I(G)+1\}$
    \end{itemize}
    
    Furthermore, if $|\mathcal{C}| \leq 2\sqrt{n}-1$, where $n$ is the number of vertices of $G$, then $\pdim I(G^\mathcal{C})=\pdim I(G)+1$.  
\end{cor}

It is worth noting that $\beta_{i,j}(\la zx \mid x \in \mathcal{C}\ra)=\displaystyle{|\mathcal{C}| \choose i+1}$ when $j=i+2$ and is zero otherwise, and this comes from the fact the the Taylor resolution on $\la zx \mid x \in \mathcal{C}\ra$ is minimal. 

\begin{proof}[Proof of Corollary \ref{splitting}]
    When $Q/I(G)$ is minimally resolved by a dg algebra, one can examine $\Cone{\bar\beta}$ from Theorem \ref{suspensionres}.  In any case, one can use the fact that
    \[
    I(G^\mathcal{C})=I(G)+\la zx \mid x \in \mathcal{C} \ra
    \]
    is a Betti splitting (see \cite{FHV}) to get
    \begin{itemize}
        \item[\textup{i)}] $\beta_{i,j}(I(G^\mathcal{C}))=\beta_{i,j}(I(G))+\beta_{i,j}(\la zx \mid x \in \mathcal{C}\ra)+\beta_{i-1,j}(\la zx \mid x \in \mathcal{C}\ra \cap I(G) \ra)$
        \item[\textup{ii)}] $\reg I(G^\mathcal{C})=\max\{\reg I(G), \reg \la zx \mid x \in \mathcal{C}\ra, \reg \la zx \mid x \in \mathcal{C}\ra \cap I(G) \ra)-1\}$
        \item[\textup{iii)}] $\pdim I(G^\mathcal{C})=\max\{\pdim I(G), \pdim \la zx \mid x \in \mathcal{C}\ra, \pdim \la zx \mid x \in \mathcal{C}\ra \cap I(G) \ra)+1\}$
    \end{itemize}
    
    Note that $\la zx \mid x \in \mathcal{C}\ra \cap I(G) \ra=zI(G)$, and so $\beta_{i-1,j}(\la zx \mid x \in \mathcal{C}\ra \cap I(G) \ra)=\beta_{i-1,j-1}(I(G))$, yielding part a).
    
    We have that $\reg \la zx \mid x \in \mathcal{C}\ra=2$, the minimum possible value for the regularity of an edge ideal.  Furthermore, $\la zx \mid x \in \mathcal{C}\ra \cap I(G) \ra=zI(G)$ has regularity $1+\reg I(G)$.  Part b) follows.    

    Finally, part c) follows from the fact that $\la zx \mid x \in \mathcal{C}\ra \cap I(G) \ra=zI(G)$ has the same projective dimension as $I(G)$ and that $Q/\la zx \mid x \in \mathcal{C} \ra$ is minimally resolved by the Taylor resolution.  

    H\'a and Hibi show that $\pdim Q/I(G) \geq 2 \sqrt n -2$, and so $\pdim I(G) \geq 2\sqrt n-3$ \cite{HH}.  If $|\mathcal{C}| \leq 2 \sqrt n -1$, then $|\mathcal{C}|-1 \leq \pdim I(G)+1$, and the ``furthermore" statement follows from part c).    
\end{proof}

It is worth noting that even if $|\mathcal{C}| > 2\sqrt n-1$, we may still have that $\pdim I(G^\mathcal{C})=\pdim I(G)+1$.  For example, take $G=C_4$ and $\mathcal{C}=V(G)$.  Then, $|\mathcal{C}|=4>2\sqrt{4}-1=3$, but $\pdim I(G^\mathcal{C})=3=2+1=\pdim I(G)+1$.  

\begin{rmk}
    In the proof of Lemma \ref{middleconeres}, we see the importance of $\tau$ producing coefficients.  We note that the above construction (Theorem \ref{suspensionres}) of the minimal free resolution of $I(G^\mathcal{C})$ still works if $\tau$ is not a dg morphism as long as the lift of $Q/I(G) \onto Q/\la x \mid x \in \mathcal{C} \ra$ to $\mathbb{G} \to \mathbb{K}$ produces coefficients in nonzero degrees (though the result that the resolution is a dg algebra does not follow).  
\end{rmk}

The next example shows that even if $\mathbb{G}$ admits the structure of a dg algebra, there is no guarantee that $\tau$ will be a dg morphism.  

\begin{ex}
\label{ex:C_5nogood}
    In this example, we show that there is no dg morphism $\tau^{C_5}$ from the minimal free resolution $\mathbb{G}^{C_5}$ of $Q/I(C_5)$ to the Koszul complex on any vertex cover of $C_5 = \{\{x,y\}, \{y,z\}, \{z,u\}, \{u,v\}, \{x,v\}\}$. We must consider the following cases determined by number of vertices in the vertex cover $\mathcal{C}$: the case where the vertex cover is minimal (containing three vertices), the case where the vertex cover contains four vertices, and the final case where the vertex cover is the entire vertex set.

    Let $\mathbb{T}^{C_5}$ be the Taylor resolution of $Q/I(C_5)$, and let $\KK^\mathcal{C}$ be the Koszul complex on $\mathcal{C}$. Since $Q/I(C_5)$ is minimally resolved by a dg algebra (see, for example, \cite[Theorem 4.1]{BE}), we know by Theorem \ref{structurethm} that $\mathbb{G}^{C_5} \cong \mathbb{T}^{C_5}/\mathbb{J}$ for some dg ideal $\mathbb{J}$ of $\mathbb{T}^{C_5}$. By Lemma \ref{Lem:tau-dgcondition},  to prove $\tau^{C_5}$ is not a dg morphism, it is enough to check that for any dg morphism $\varphi: \mathbb{T}^{C_5} \to \KK^\mathcal{C}$ (see Lemma \ref{phi}), the dg ideal $\mathbb{J}$ is not contained in the kernel of $\varphi$.
    
    We begin with the minimal vertex covers. By symmetry, we need only consider one: take $\mathcal{C} = \{x,z,v\}$.  There are only two options for the map $\varphi$: either $\varphi=\varphi_{1}$ sending $xv$ to $xe_v$ or $\varphi=\varphi_2$ sending $xv$ to $ve_x$.  \looseness -1

    Using discrete Morse theory, a dg ideal $\mathbb{J}$ such that $\mathbb{G}^{C_5} \cong \mathbb{T}^{C_5}/\mathbb{J}$ may be computed explicitly (see \cite[Example 2.31]{GMP}). In doing so, we notice that any such dg ideal $\mathbb{J}$ must contain both $g_{\{xy, yz, xv\}}$ and $ g_{\{zu,uv,xv\}}$ (there are other basis elements where this is also true, but these are of particular interest). Now, $\varphi_1\left(g_{\{xy, yz, xv\}}\right) = ye_{\{x,z,v\}}$ which is nonzero in $\KK^\mathcal{C}$. Similarly, $\varphi_2\left(g_{\{zu,uv,xv\}}\right) = ue_{\{x,z,u\}}$ is nonzero in $\KK^\mathcal{C}$. Thus, $\mathbb{J}$ is not contained in the kernel of either map, meaning there is no dg morphism $\tau^{C_5}$ from $\mathbb{G}^{C_5}$ to the Koszul complex on any minimal vertex cover of $C_5$.

Next we consider vertex covers containing four vertices. Similar to the previous case, symmetry allows us to consider one vertex cover $\mathcal{C} = \{x,y,z,u\}$. Again by symmetry, there are only two maps (up to isomorphism) $\varphi = \varphi_1$ and $\varphi = \varphi_2$: 
\vspace{-3mm}
\begin{equation}
  \begin{minipage}[t]{0.45\textwidth}
    \begin{align*}
    \varphi_1: \mathbb{T}^{C_5} &\to \KK^\mathcal{C} \\
    g_{xy} &\mapsto ye_x\\
     g_{yz} &\mapsto ze_y\\
      g_{zu} &\mapsto ue_z\\
     g_{uv} &\mapsto ve_u\\
      g_{xv} &\mapsto ve_x
\end{align*}
  \end{minipage}%
  \hfill
  \begin{minipage}[t]{0.45\textwidth}
    \begin{align*}
   \varphi_2: \mathbb{T}^{C_5} &\to \KK^\mathcal{C}\\
    g_{xy} &\mapsto xe_y\\
     g_{yz} &\mapsto ze_y\\
      g_{zu} &\mapsto ue_z\\
     g_{uv} &\mapsto ve_u\\
      g_{xv} &\mapsto ve_x
\end{align*}
  \end{minipage}
\end{equation}
\vspace{1mm}

Recall that any dg ideal $\mathbb{J}$ must contain both $g_{\{xy, yz, xv\}}$ and $ g_{\{zu,uv,xv\}}$.  It is not hard to see that $\varphi_1(g_{\{zu,uv,xv\}}) =ve_{\{x,z,u\}}$, which is not zero in $\KK^{\mathcal{C}}$. Similarly, $\varphi_2(g_{\{zu,uv,xv\}}) = ve_{\{x,z,u\}}$. Thus, our dg ideal $\mathbb{J}$ is not contained in the kernel of either $\varphi$, so there is no dg morphsim $\tau^{C_5}$ from $\mathbb{G}^{C_5}$ to the Koszul complex on any vertex cover of $C_5$ containing four vertices.

    Lastly, we must consider when our vertex cover is the entire vertex set, i.e. $\mathcal{C} = \{x,y,z,u,v\}$. Up to isomorphism, there is only one map $\varphi$ to consider:
    \begin{align*}
    \varphi: \mathbb{T}^{C_5} &\to \KK^\mathcal{C}\\
    g_{xy} &\mapsto ye_x\\
     g_{yz} &\mapsto ze_y\\
      g_{zu} &\mapsto ue_z\\
     g_{uv} &\mapsto ve_u\\
      g_{xv} &\mapsto xe_v.
\end{align*}
Once again, $\varphi(g_{\{zu,uv,xv\}}) = uxe_{\{z,u,v\}}$ is not zero in $\KK^{\mathcal{C}}$.  Alternatively, note that $D(C_5,\mathcal{C})$ is an oriented cycle (see Lemma \ref{mintau}). Thus, there is no dg morphism $\tau^{C_5}$ from $\mathbb{G}^{C_5}$ to the Koszul complex on any vertex cover of $C_5$.
\end{ex}

With the revelation that not all graphs, even dg graphs (those with edge ideals minimally resolved by dg algebras), may be used to build other dg graphs via this method of suspension, there comes the question of what graphs can be used. In other words, for which graphs and vertex covers is $\tau$ a dg morphism? The next proposition shows that the property of $\tau$ being a dg morphism is hereditary, meaning, in particular, that $\tau$ cannot be a dg morphism for any graph containing $C_5$ as an induced subgraph.

\begin{prop}
\label{prop:Hereditarydgtau}
    Let $H$ be an induced dg subgraph of a dg graph $G$. If there exists a dg morphism $\tau^G$ from the minimal free resolution of $Q/I(G)$ to the Koszul complex on any vertex cover $\mathcal{D}$ of $G$, then there exists a dg morphism $\tau^H$ from the minimal free resolution of $Q/I(H)$ to the Koszul complex on any vertex cover $\mathcal{C} \subseteq \mathcal{D}$ of $H$. 
\end{prop}

\begin{proof}
By assumption, we know there exist dg morphisms $\varphi^G: \mathbb{T}^G \to \KK^\mathcal{D}$ and $\tau^G:\mathbb{G}^G \to \KK^\mathcal{D} $such that $\tau^G\circ\pi^G = \varphi^G$. By the pruning process detailed in \cite{boocher}, we may prune the resolutions $\mathbb{T}^G, \mathbb{G}^G,\KK^G$ to $\mathbb{T}^H, \mathbb{G}^H,\KK^H$, respectively by the maps $\alpha, \beta,$ and $\gamma$. This process preserves dg-ness \cite{GMP}.

Since $\alpha$ is a surjection, we may define the map \[\varphi^H: \mathbb{T}^H \to \KK^\mathcal{C}\] by lifting $Q/I(H) \onto Q/\la x \mid x \in \mathcal{C}\ra$ so that $\varphi^H \circ\alpha = \gamma \circ\varphi^G$. Similarly, using the projection map $\pi^H: \mathbb{T}^H \to \mathbb{G}^H$, we may define the map \[ \tau^H : \mathbb{G}^H \to \KK^\mathcal{C}\] so that $\tau^H \circ\pi^H = \varphi^H$. To show $\tau^H$ is a dg morphism, it suffices to show $\varphi^H$ is a dg morphism and the following diagram commutes (Lemma \ref{Lem:tau-dgcondition}).

\begin{figure}[H]
\begin{tikzcd}
                                        & \mathbb{T}^G \arrow[ddd,"\pi^G"] \arrow[rrrddd,"\varphi^G"] \arrow[ld,"\alpha"] &  &              &                         \\
\mathbb{T}^H \arrow[ddd,"\pi^H"] \arrow[rrrddd, crossing over, "\varphi^H"] &                                                    &  &              &                         \\
                                        &                                                    &  &              &                         \\
                                        & \mathbb{G}^G \arrow[ld,"\beta"] \arrow[rrr,"\tau^G"]                &  &              & \mathbb{K}^\mathcal{D} \arrow[ld,"\gamma"] \\
\mathbb{G}^H \arrow[rrr,"\tau^H"]                &                                                    &  & \mathbb{K}^\mathcal{C} &                        
\end{tikzcd}
\caption{Pruning $\tau^G$ to $\tau^H$}
\label{fig:cheesewedge}
\end{figure}
The fact that $\varphi^H$ is an algebra map follows from the fact that $\varphi^G$ is an algebra map and also from the fact that the pruning maps are dg morphisms:
\begin{align*}
    \varphi^H(\overline{g_1g_2}) = \overline{\varphi^G(g_1g_2)} = \overline{\varphi^G(g_1)\varphi^G(g_2)} = \overline{\varphi^G(g_1)}\ \overline{\varphi^G(g_2)} = \varphi^H(\overline{g_1})\varphi^H(\overline{g_2}).
\end{align*}

To show $\varphi^H$ is a chain map, we utilize the right cancellation property for surjections. Indeed,
\begin{align*}
\varphi^H \circ \partial^{\mathbb{T}^H}\circ \alpha &= \varphi^H \circ \alpha \circ \partial^{\mathbb{T}^G}\\ &= \gamma \circ \varphi^G \circ \partial^{\mathbb{T}^G}\\
&=\gamma \circ \partial^{\KK^G} \circ \varphi^G \\
&= \partial^{\KK^H}\circ \gamma \circ \varphi^G\\
&= \partial^{\KK^H}\circ \varphi^H \circ \alpha.
\end{align*}
Thus, $\varphi^H \circ \partial^{\mathbb{T}^H} = \partial^{\KK^H} \circ\varphi^H $, which completes the proof that $\varphi^H$ is a dg morphism.

To show the rest of the diagram in Figure \ref{fig:cheesewedge} commutes, we begin by referring to \cite[Theorem 5.7, Step 3]{GMP} which shows $\alpha \circ \pi^H = \pi^G \circ \beta$ by the nature of pruning. Finally, consider
\[
    \tau^H \circ \beta \circ \pi^G = \tau^H \circ \pi^H \circ \alpha
    = \varphi^H \circ \alpha 
    = \gamma \circ \varphi^G  = \gamma \circ \tau^G \circ \pi^G.
\]
Since $\pi^G$ is a surjection, we right cancel to get $\tau^H \circ \beta = \gamma \circ \tau^G$. 
\end{proof}

\begin{remark}
    Example \ref{ex:C_5nogood} shows there is no dg morphism $\tau^{C_5}$ from the minimal free resolution $\mathbb{G}^{c_5}$ of $Q/I(C_5)$ to the Koszul complex $\KK^\mathcal{C}$ on any vertex cover $\mathcal{C}$ of $C_5$. By Proposition \ref{prop:Hereditarydgtau}, any graph $G$ that has $C_5$ as an induced subgraph cannot have a dg morphism $\tau^G$ from the minimal free resolution $\mathbb{G}^G$ of $Q/I(G)$ to the  Koszul complex $\KK^\mathcal{D}$ on any vertex cover $\mathcal{D}$ of $G$.
\end{remark}

\subsection{Iterative construction of DG graphs}\label{iterativeconstructionsection}

We now apply the construction in Section \ref{suspensionsection} to produce infinitely many graphs whose edge ideals are minimally resolved by dg algebras.  In particular, threshold graphs fall into this category.  

\begin{thm}\label{iterate}
    Let $\{G^i\}_{i\geq 0}$ be a family of finite simple graphs.  For each $i$, let $\mathcal{C}^i$ be a vertex cover of $G^i$.  Suppose
\begin{enumerate}
    \item $V(G^0)$ is nonempty 
    \item $D(G^0,\mathcal{C}^0)$ has no oriented cycles
    \item $V(G^i) \cup \{z_{i+1}\} \subseteq V(G^{i + 1})$
    \item $\mathcal{C}^i \subseteq \mathcal{C}^{i+1}$
    \item $E(G^{i+1}) = E(G_i) \cup \{z_{n+1}x : x \in \mathcal{C}^i \}$
\end{enumerate}
Let $\KK^i$ be the Koszul complex on the vertex cover $\mathcal{C}^i$.  If the minimal free resolution $\GG^0$ of $Q/I(G^0)$ is a dg algebra and the map $\tau^{0}:\GG^0 \to \KK^0$ lifting the surjection $Q/I(G^0) \onto Q/\la x \mid x \in \mathcal{C}^0 \ra$ is a dg morphism, then the minimal free resolution $\GG^i$ of $Q/I(G^i)$ is a dg algebra for all $i \geq 0$.  
\end{thm}

\begin{proof} 
We begin by fixing the following notation for resolutions. For $n \geq 0$, let
\begin{enumerate}
    \item $\GG^i$ be the minimal free resolution of $Q/I(G^i)$
    \item $\F^{i} = \mathbb{T}(z_{i+1} \mathcal{C}^i)$ be the Taylor resolution on the edges $z_{i+1} \mathcal{C}^i:= \{z_{i+1}x : x\in \mathcal{C}^i \}$.
\end{enumerate}

If $\GG^i$ is a dg algebra and\[\tau^i: \GG^i \to \KK^i\]
as defined in Lemma \ref{tau} is a dg morphism, an application of Lemma \ref{newmodstructure} and Theorem \ref{suspensionres} shows that $\text{Cone}(\overline{\beta}) \cong \GG^{i+1}$ is a dg algebra. Thus, it suffices to show that one may construct a dg morphism $\tau^{i+1}: \GG^{i+1} \to \KK^{i+1}$ from $\tau^i$. First, consider the map
\begin{align*}
    \iota^{i}: \F^{i} &\longrightarrow \KK^{i} \\
        1 &\longmapsto 1 \\
        f_{z_{i+1}x_{j_1},\ldots,z_{i+1}x_{j_k}} &\longmapsto z_{i+1}e_{x_{j_1},\ldots,{x_{j_k}}}
\end{align*}

It is easy to show $\iota^i$ is an algebra map. In the case that $x_{j_s}\neq x_{\ell_t}$ for all $s,t$, the image of the product
\[
f_{z_{i+1}x_{j_1},\ldots,z_{i+1}x_{j_k}} \cdot f_{z_{i+1}x_{\ell_1},\ldots,z_{i+1}x_{\ell_m}}
\]
under $\iota^i$ is
\[
(-1)^\sigma z_{i+1}^2e_{x_{j_1},\ldots,x_{j_k},x_{\ell_1},\ldots,x_{\ell_m}},
\]
and
\begin{align*}
\iota^i(f_{z_{i+1}x_{j_1},\ldots,z_{i+1}x_{j_k}}) \cdot \iota^i(f_{z_{i+1}x_{\ell_1},\ldots,z_{i+1}x_{\ell_m}})&=z_{i+1}^2e_{x_{j_1},\ldots,x_{j_k}}\cdot e_{x_{\ell_1},\ldots,x_{\ell_m}}\\
&=(-1)^\sigma z_{i+1}^2e_{x_{j_1},\ldots,x_{j_k},x_{\ell_1},\ldots,x_{\ell_m}},
\end{align*}
and when there exist $s,t$ with $x_{j_s}=x_{\ell_t}$, 
\[
\iota^i(f_{z_{i+1}x_{j_1},\ldots,z_{i+1}x_{j_k}}\cdot f_{z_{i+1}x_{\ell_1},\ldots,z_{i+1}x_{\ell_m}})=0=\iota^i(f_{z_{i+1}x_{j_1},\ldots,z_{i+1}x_{j_k}}) \cdot \iota^i(f_{z_{i+1}x_{\ell_1},\ldots,z_{i+1}x_{\ell_m}}).
\]

Above, $\sigma=\sigma(V,W)$ with $V=\{z_{i+1}x_{j_1},\ldots,z_{i+1}x_{j_k}\}$ and $W=\{z_{i+1}x_{\ell_1},\ldots,z_{i+1}x_{\ell_m}\}$ (by earlier choice of the total order $<$, the $z_{i+1}$ may be omitted everywhere in $V$ and $W$ without the sign changing).  

To see $\iota^i$ is a chain map, we let  $V = \{z_{i+1}x_{j_1},\ldots,z_{i+1}x_{j_k}\}$ and $V' = \{x_{j_1},\ldots,x_{j_k}\}$ and consider 
\[
\partial^{\KK^i}\iota^i(f_{V}) = \partial^{\KK^i} \left( z_{i+1}e_{V'}\right) = z_{i+1}\sum_{t =1}^{k}(-1)^{t-1}x_{j_t}e_{V'\smallsetminus\{x_{j_t}\}}.
\]
On the other hand
\begin{align*}
\iota^i \partial^{\FF^i}(f_{V}) &= \iota^i\left( \sum_{t=1}^{k}\dfrac{z_{i+1}\lcm(V')}{z_{i+1}\lcm(V'\smallsetminus\{x_{j_t}\})}f_{V\smallsetminus\{x_{j_t}\}}\right)\\
&= \iota^i \left( \sum_{t=1}^{k} x_{j_t}f_{V\smallsetminus\{x_{j_t}\}}\right)\\
= &
z_{i+1}\sum_{t =1}^{k}(-1)^{t-1}x_{j_t}e_{V'\smallsetminus\{x_{j_t}\}},
\end{align*}
showing $\iota^i$ is a chain map and thus a dg morphism.

Now define the new graph $\tilde{G}^i$ as the suspension of $z_{i+1}$ over the vertex cover $\mathcal{C}^i$ of $G^i$, and let $\tilde{\GG}^i$ be the minimal free resolution of $Q/I(\tilde{G}^i)$. Define 
    \begin{align*}
    \tilde\tau^{i}: \tilde{\GG}^{i} &\longrightarrow \KK^{i} \\
        (a,m) & \mapsto \begin{cases} \tau^i(a) + \iota^i(b), & m = (b,c) \in M_{\geq 1}^{\natural}\\
        \tau^i(a)+\iota^i(m), & m\in M_0=B_1=\mathbb{F}^i_1
            \end{cases}
\end{align*}
for $a\in A_{\geq 1}^{\natural}=\left(\GG^i\right)_{\geq 1}^{\natural}$.  Here, $M$ is $\Cone{\alpha^i_{\geq 1}}[1]$, where $\alpha^i:\mathbb{G}^i \to \FF^i$ is the composition of $\tau^i$ and $\mu^z$ (see Lemma \ref{mulem}).
Let $G^{i+1}$ be the graph obtained from $\tilde{G}^i$ by introducing any number of isolated vertices. Since isolated vertices do not contribute to the edge ideal of the graph and adding an isolated vertex corresponds to tensoring to a larger ring (which preserves dg structure), we have $\tilde{\GG}^{i} \cong \GG^{i+1}$ and can define
   $ \tau^{i+1}: \GG^{i+1} \to \KK^{i+1} $ by setting $\tau^{i+1} = \tilde\tau^i$, where the images of both $\tau^i$ and $\iota^{i}$ are understood to be their images under the inclusion $\KK^i \subseteq \KK^{i+1}$.

    Under the assumption that $\tau^{i}$ is a dg morphism, we show that $\tau^{i+1}$ is a dg morphism.   
We first note that the following square commutes, since $I(G^{i+1}) \subseteq \langle x \mid x\in \mathcal{C}^{i+1}\rangle $:
 \[
\begin{tikzcd}
\mathbb{G}^{i+1}_0 = \mathbb{F}^i_{0}\oplus \mathbb{G}^i_{0} \arrow[r, two heads] \arrow[d, "\iota^i_{0}=\text{id}"] & Q/I(G^{i+1}) \arrow[d, two heads] \\
\mathbb{K}^i_0 \arrow[r, two heads]                        & Q/\langle x \mid x \in \mathcal{C}^{i+1} \rangle        
\end{tikzcd}
    \]

    Since both $\tau^{i}$ and $\iota^{i}$ are both chain maps, paired with the fact $\iota^i \alpha = z \tau^i$, it is easy to see that $\tau^{i+1}$ is also a chain map. Let $a\in A_{\geq1}^{\natural}$ and $m = (b,c) \in M_{\geq 1}^{\natural}$, then
    \begin{align*}
        \tau^{i+1}\partial^{\text{Cone}(\overline{\beta})}(a,m) &= \tau^{i+1}\left(\partial^{\GG}(a)+\overline{\beta}(m), -\partial^M(m)\right) \\
        &= (\tau^{i}+\iota^i)\left(\partial^{\GG}(a)-zc, \left(\partial^{\FF}(b)+\alpha(c), -\partial^{\GG}(c)\right) \right)\\
        &= \tau^i\left( \partial^{\GG}(a)\right) - z\tau^i(c) + \iota^i\left(\partial^{\FF}(b)\right) + \iota^i(\alpha(c))\\
        &=\tau^i\left( \partial^{\GG}(a)\right) - z\tau^i(c) + \iota^i\left(\partial^{\FF}(b)\right) + z\tau^i(c)\\
        &=\tau^i\left( \partial^{\GG}(a)\right)+\iota^i\left(\partial^{\FF}(b)\right)\\
        & = \partial^{\KK^i}\tau^{i}(a) + \partial^{\KK^i}\iota^{i}(b)\\
        & = \partial^{\KK^i}\tau^{i+1}(a,m).
    \end{align*}

When $m \in M_0$, one has
\begin{align*}
    \tau^{i+1}\partial^{\text{Cone}(\overline{\beta})}(a,m) &= \tau^{i+1}\left(\partial^{\GG}(a)+\overline{\beta}(m), -\partial^M(m)\right) \\
    &= \left(\tau^i + \iota^i \right)\left(\partial^{\GG}(a)+\partial^{\FF}(m), 0\right)\\
    &= \tau^i(\partial^{\GG}(a))+\tau^i(\partial^{\FF}(m))\\
    &= \tau^i(\partial^{\GG}(a))+\iota^i(\partial^{\FF}(m))\\
    & = \partial^{\KK^i}\tau^{i}(a) + \partial^{\KK^i}\iota^{i}(m)\\
    &=\partial^{\KK^i}\tau^{i+1}(a,m).
\end{align*}

 To prove that $\tau^{i+1}$ is an algebra map, we consider the following cases:

\noindent CASE 1: $m = (b,c), n=(d,e) \in M_{\geq 1}^{\natural} $ and $a,f \in A_{\geq 1}^{\natural}$
\begin{align*}
\tau^{i+1}((a,m)(f,n)) &= \tau^{i+1}\left( af, mf + (-1)^{|a|}an + (-1)^{|m|}\mu(m \otimes n)\right)\\
&= \tau^{i+1}\left(af, \left( \frac{1}{z}b\alpha(f),cf\right) + (-1)^{|a|}\left((-1)^{|a|}\frac{1}{z}\alpha(a)d,ae \right) + (bd,0)\right)\\
&= \tau^{i}(af) + \iota^i\left(\frac{1}{z}b\alpha(f) +\frac{1}{z}\alpha(a)d + bd \right)\\
&= \tau^{i}(af) + \iota^i\left(\frac{1}{z}b\alpha(f)\right)+\iota^i\left(\frac{1}{z}\alpha(a)d\right)+\iota^i(bd)\\
& = \tau^{i}(af) + \frac{1}{z}\iota^i(b) \iota^i(\alpha(f)) +\frac{1}{z}\iota^i(\alpha(a))\iota^i(d)+\iota^i(bd)\\
& = \tau^{i}(a)\tau^{i}(f)+\iota^i(b)\tau^{i}(f) + \tau^{i}(a)\iota^i(d)+\iota^i(b)\iota^i(d)\\
&= (\tau^{i}(a)+\iota^{i}(b))(\tau^{i}(f)+\iota^{i}(d))\\
    &= \tau^{i+1}(a,m)\tau^{i+1}(f,n)
\end{align*}

\noindent CASE 2: $m \in M_0$, $n=(d,e) \in M_{\geq 1}^{\natural}$, and $a,f \in A_{\geq 1}^{\natural}$
\begin{align*}
\tau^{i+1}((a,m)(f,n)) &= \tau^{i+1}\left( af, mf + (-1)^{|a|}an + (-1)^{|m|}\mu(m \otimes n)\right)\\
&= \tau^{i+1}\left(af, \left( \frac{1}{z}m\alpha(f),-\frac{1}{z}f\partial^B(m)\right) + (-1)^{|a|}\left((-1)^{|a|}\frac{1}{z}\alpha(a)d,ae \right) + (md,0)\right)\\
&= \tau^{i}(af) + \iota^i\left(\frac{1}{z}m\alpha(f) +\frac{1}{z}\alpha(a)d + md \right)\\
&= \tau^{i}(af) + \iota^i\left(\frac{1}{z}m\alpha(f)\right)+\iota^i\left(\frac{1}{z}\alpha(a)d\right)+\iota^i(md)\\
& = \tau^{i}(af) + \frac{1}{z}\iota^i(m) \iota^i(\alpha(f)) +\frac{1}{z}\iota^i(\alpha(a))\iota^i(d)+\iota^i(md)\\
& = \tau^{i}(a)\tau^{i}(f)+\iota^i(m)\tau^{i}(f) + \tau^{i}(a)\iota^i(d)+\iota^i(m)\iota^i(d)\\
&= (\tau^{i}(a)+\iota^{i}(m))(\tau^{i}(f)+\iota^{i}(d))\\
    &= \tau^{i+1}(a,m)\tau^{i+1}(f,n)
\end{align*}

\noindent CASE 3:  $m=(b,c) \in M_{\geq 1}^{\natural}$, $n \in M_0$, and $a,f \in A_{\geq 1}^{\natural}$
\begin{align*}
\tau^{i+1}((a,m)(f,n)) &= \tau^{i+1}\left( af, mf + (-1)^{|a|}an + (-1)^{|m|}\mu(m \otimes n)\right)\\
&= \tau^{i+1}\left(af, \left( \frac{1}{z}b\alpha(f),cf\right) + (-1)^{|a|}\left((-1)^{|a|}\frac{1}{z}\alpha(a)n,-\frac{1}{z}a \partial^B(n) \right) + (bn,0)\right)\\
&= \tau^{i}(af) + \iota^i\left(\frac{1}{z}b\alpha(f) +\frac{1}{z}\alpha(a)n + bn \right)\\
&= \tau^{i}(af) + \iota^i\left(\frac{1}{z}b\alpha(f)\right)+\iota^i\left(\frac{1}{z}\alpha(a)n\right)+\iota^i(bn)\\
& = \tau^{i}(af) + \frac{1}{z}\iota^i(b) \iota^i(\alpha(f)) +\frac{1}{z}\iota^i(\alpha(a))\iota^i(n)+\iota^i(bn)\\
& = \tau^{i}(a)\tau^{i}(f)+\iota^i(b)\tau^{i}(f) + \tau^{i}(a)\iota^i(n)+\iota^i(b)\iota^i(n)\\
&= (\tau^{i}(a)+\iota^{i}(b))(\tau^{i}(f)+\iota^{i}(n))\\
    &= \tau^{i+1}(a,m)\tau^{i+1}(f,n)
\end{align*}

CASE 4: $m, n \in M_0$ and $a,f \in A_{\geq 1}^{\natural}$
\begin{align*}
\tau^{i+1}((a,m)(f,n)) &= \tau^{i+1}\left( af, mf + (-1)^{|a|}an + (-1)^{|m|}\mu(m \otimes n)\right)\\
&= \tau^{i+1}\left(af, \left( \frac{1}{z}m\alpha(f),-\frac{1}{z}f\partial^B(m)\right) + \left(\frac{1}{z}\alpha(a)n,-(-1)^{|a|}\frac{1}{z}a \partial^B(n) \right) + (mn,0)\right)\\
&= \tau^{i}(af) + \iota^i\left(\frac{1}{z}m\alpha(f) +\frac{1}{z}\alpha(a)n + mn \right)\\
&= \tau^{i}(af) + \iota^i\left(\frac{1}{z}m\alpha(f)\right)+\iota^i\left(\frac{1}{z}\alpha(a)n\right)+\iota^i(mn)\\
& = \tau^{i}(af) + \frac{1}{z}\iota^i(m) \iota^i(\alpha(f)) +\frac{1}{z}\iota^i(\alpha(a))\iota^i(n)+\iota^i(mn)\\
& = \tau^{i}(a)\tau^{i}(f)+\iota^i(m)\tau^{i}(f) + \tau^{i}(a)\iota^i(n)+\iota^i(m)\iota^i(n)\\
&= (\tau^{i}(a)+\iota^{i}(m))(\tau^{i}(f)+\iota^{i}(n))\\
    &= \tau^{i+1}(a,m)\tau^{i+1}(f,n)
\end{align*}
    Thus, $\tau^{i+1}$ is a dg morphism.  Minimality follows from the facts that $D(G^0,\mathcal{C}^0)$ has no oriented cycles (see Lemma \ref{mintau}) and that if $\tau^i$ always produces coefficients, then so does $\tau^{i+1}$: the edges involving the new vertex $z_{i+1}$ are always oriented away from $z_{i+1}$, and so no new directed cycle could be created in the suspension.
\end{proof}

As a final point, we note that the above iterative construction allows us to conclude that threshold graphs have edge ideals minimally resolved by dg algebras.  A \textbf{threshold graph} is a graph obtained from the empty graph by performing a sequence made up of the following two operations:
\begin{itemize}
    \item[1)] introduce an isolated vertex to the current graph $G$, obtaining $G_*$,
    \item[2)] suspend a new vertex over the entire vertex set of the current graph $G$, obtaining $G^*$.
\end{itemize}

\begin{ex}
    The threshold graph $\phantom{}_{*\,*\,\,\,*}^{\,\,\,\,\,\,\,*\,\,\,*}$ is given in Figure \ref{thresholdex}, in which the vertices are labeled in the order in which they were introduced.
    \begin{figure}[h]
        \centering
\begin{tikzpicture}[scale=1.25,
    every node/.style={
        circle,
        draw,
        fill=white,
        inner sep=0.8pt,
        font=\scriptsize
    }]

\node (v1) at (0,0) {$1$};
\node (v2) at (1,0) {$2$};
\node (v3) at (0,-1) {$3$};
\node (v4) at (1,-1) {$4$};
\node (v5) at (0.5,-1.5) {$5$};

\draw (v1) -- (v3);
\draw (v2) -- (v3);
\draw (v3) -- (v5);
\draw (v4) -- (v5);
\draw (v2) -- (v5);
\draw (v1) -- (v5);

\end{tikzpicture}
        \caption{A threshold graph}
        \label{thresholdex}
    \end{figure}
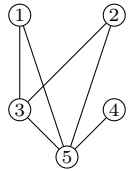
\end{ex}

\begin{cor}\label{threshold}
    If $G$ is a threshold graph, then the minimal free resolution of $Q/I(G)$ admits the structure of a dg algebra.
\end{cor}

\begin{proof}
    By Theorem \ref{iterate}, it suffices to show that if $G$ is a collection of isolated vertices and $\mathcal{C}$ is the entire vertex set of $G$, then $\tau^G$ is a dg morphism.  This is trivially true, since $\tau^G$ is the zero map in positive degrees.  
\end{proof}

\printbibliography

\end{document}